\documentclass[10pt]{amsart}
\usepackage{pdflscape}
\usepackage{preample}

\begin{document}
%\begin{landscape}

\title{Hodge cohomology with a ramification filtration, II}
\date{\today}
\author[S. Kelly]{Shane Kelly}
\address{%Department of Mathematics, 
Graduate School of Mathematical Sciences
University of Tokyo
3-8-1 Komaba Meguro-ku
Tokyo 153-8914, Japan
}
\email{shanekelly@g.ecc.u-tokyo.ac.jp}
\author[H. Miyazaki]{Hiroyasu Miyazaki}
\address{NTT Corporation, NTT Institute for Fundamental Mathematics, 3-1 Morinosato-Wakamiya,Atsugi,Kanagawa 243-0198, Japan}
\email{hiroyasu.miyazaki@ntt.com}
\thanks{
The first author was supported by JSPS KAKENHI Grant (19K14498). %
The second author is supported by JSPS KAKENHI Grant (21K13783).
}

\begin{abstract}
As a sequel of Part I, we consider a filtration of Hodge cohomology groups indexed by divisors ``at infinity'', and prove that they are represented in the category of motives with modulus. 
In particular, we obtain a realisation functor of the Hodge cohomology groups. 
\end{abstract}

\maketitle

\begin{center}
\today
\end{center}

%\vspace{1cm}

\setcounter{tocdepth}{1}
\tableofcontents

\section{Introduction} \label{sec:intro}

The theory of mixed motives developed by Voevodsky can be regarded as a theory which unifies the cohomologies of $\AA^1$-homotopy invariant sheaves, including Betti, de Rham and $\ell$-adic \'etale cohomolgies.
Voevodsky's theory has provided many fruitful applications, but it cannot capture non-$\mathbb{A}^1$-invariant cohomologies. 
A very standard example of such is the cohomology of the structure sheaf $\OO$, which is obviously far from $\mathbb{A}^1$-invariant: $\OO_{X \times \AA^1} \cong \OO_X [t] \neq \OO_X$. 

In the previous work \cite{KelMi1}, we proved that the cohomology of the structure sheaf $\OO$ can be represented in the category of motives with modulus $\ulMDM^{\eff}$. 
The category was introduced by Kahn, Miyazaki, Saito, and Yamazaki in \cite{kmsy1}, \cite{kmsy2}, \cite{kmsy3} as an upgrade of Voevodsky's category of mixed motives. 
Let us recall the precise statement:

\begin{thm}\label{thm:main-I}
Suppose $\mathrm{char} (k)=0$.
Then there exists a unique object $\ulMO \in \ulMDM^{\eff}_k$ such that for any $\sX = (\hX,\mX) \in \ulMCor$ with normal crossings and for any $n \in \Z$ we have 
\[
\Hom_{\ulMDM^{\eff}_k} (M(\sX),\ulMO[n]) \cong H^n_{\Zar} (\hX,\sqrt{I} \otimes I^{-1}),
\]
where $I$ denotes the ideal sheaf defining $\mX \subset \hX$.
\end{thm}

Taking $(\hX,\mX) = (X,\varnothing)$, we immediately obtain that the cohomology of the structure sheaf $\OO$ is also represented in the category of motives with modulus:

\begin{cor}
Suppose $\mathrm{char} (k)=0$.
Then for any $n \in \Z$ and $X \in \Sm$ we have
\[
\Hom_{\ulMDM^{\eff}_k} (M(X,\varnothing),\ulMO[n]) \cong H^n_{\Zar} (X,\OO).
\]
\end{cor}

The aim of this paper is to prove the same results for the Hodge sheaves $\Omega^q$, which are of course not $\AA^1$-homotopy invariant. 
More precisely, we construct a modulus Hodge sheaf $\ulMOmega^q$ for all non-negative integers $q$ by combining $\ulMO$ and the logarithmic Hodge sheaf. 
Our main result is the following. 

\begin{thm}\label{thm:main-II}
Suppose $\mathrm{char} (k)=0$. 
Then, for any non-negative integer $q$, there exists a unique object $\RMOmega^q \in \ulMDM^{\eff}_k$ such that for any $\sX = (\hX,\mX) \in \ulMCor$ with normal crossings we have 
\[
\Hom_{\ulMDM^{\eff}_k} (M(\sX),\RMOmega^q[n]) \cong H^n_{\Zar} (\hX,\ulMOmega^q_{\sX}),
\]
Here, the sheaf $\ulMOmega^q_{\sX}$ on the small site $\hX_{\Zar}$ is defined by 
\[
\ulMOmega^q_{\sX} := \Omega_X (\log |\mX|) \otimes \sqrt{I} \otimes I^{-1},
\]
where $I$ denotes the ideal sheaf defining $\mX \subset \hX$, and $\Omega_X (\log |\mX|)$ denotes the sheaf of differential forms having at most logarithmic poles along the support $|\mX|$ of the divisor $\mX$.
\end{thm}

Notice that $\ulMOmega^q_{(\hX,m\mX)} \subset  \ulMOmega^q_{(\hX,(m+1)\mX)}$ for all $m \geq 1$ and 
\[
\colim_{m \geq 1} \ulMOmega^q_{(\hX,m\mX)} \xrightarrow{\sim} \Omega^q_{X},
\]
which means that $\ulMOmega^q$ gives an ``exhaustive filtration'' of $\Omega^q$.
In fact, we have an isomorphism of the same form between the cohomology groups:
\[
\colim_{m \geq 1} H^n_{\Zar} (\hX,\ulMOmega^q_{(\hX,m\mX)}) \xrightarrow{\sim} H^n_{\Zar} (X,\Omega^q_{X}).
\]

As before, taking $(\hX,\mX) = (X,\varnothing)$, we obtain the Hodge realisation functor:

\begin{cor}
Suppose $\mathrm{char} (k)=0$.
Then, for any non-negative integer $q$ and $n \in \Z$, we have 
\[
\Hom_{\ulMDM^{\eff}_k} (M(X,\varnothing),\RMOmega^q[n]) \cong H^n_{\Zar} (X,\Omega^q_{X}).
\]
\end{cor}

Also, taking $(\hX,\mX) = (\hX,|\mX|)$, we obtain the logarithmic Hodge realisation functor:

\begin{cor}
Suppose $\mathrm{char} (k)=0$.
Then, for any non-negative integer $q$ and $n \in \Z$, we have 
\[
\Hom_{\ulMDM^{\eff}_k} (M(\hX,|\mX|),\RMOmega^q[n]) \cong H^n_{\Zar} (\hX,\Omega^q_{\hX} (\log |\mX|)).
\]
\end{cor}

Also, we can construct a monoidal ``de Rham realisation with modulus'' which extends the classical de Rham realisation functor. 
Let $\ulMDM^{\eff}_{\gm} \subset \ulMDM^{\eff}$ be the full subcategory consisting of compact objects (cf. \cite[Theorem 3.3.1]{kmsy3}). Recall from \cite[Lemma 6.3.3, Corollary 6.3.4]{kmsy3} that there exists a fully faithful monoidal functor $\DM^{\eff}_{\gm} \to \ulMDM^{\eff}_{\gm}$, under the existence of resolution of singularities. 

\begin{thm}[Theorem \ref{thm:monoidal}, Theorem \ref{thm:LWcomparison}, Remark \ref{rem:dervec}]\label{main:monoidal}
Suppose $\mathrm{char} (k)=0$.
Consider the triangulated functor
\[
R_{dR} : \ulMDM^{\eff,\op}_{\gm} \to D(\mathrm{Vect}_k) ; \quad M \mapsto \Hom_{\ulMDM^{\eff}}(-,\oplus_{q \geq 0}\RMOmega^q[-q]),
\]
where $D(\mathrm{Vect}_k)$ denotes the derived category of the category of vector spaces over $k$, equipped with the monoidal structure that is induced by the usual tensor product of $k$-vector spaces. 
Then $R_{dR}$ is monoidal, and the composite 
\[
\DM^{\eff,\op}_{\gm} \to \ulMDM^{\eff,\op}_{\gm} \xrightarrow{R_{dR}} D(\mathrm{Vect}_k)
\]
coincides with the de Rham realisation functor.
\end{thm}

As an immediate consequence of Theorem \ref{main:monoidal}, we obtain the following corollary. 
We hope that it will give us some hints about which object should be inverted to obtain a category of motives with modulus admitting duality.  

\begin{cor}[Corollary \ref{cor:dualisable}] \label{main:dualisable}
Let $T \in \ulMDM^{\eff}_{\gm}$ be an object such that $R_{dR}$ factors through $\ulMDM^{\eff}_{\gm}[T^{-1}]$.  
Then $R_{dR}(T)$ is one dimensional. 
If $M \in \ulMDM^{\eff}_{\gm}$ is an object such that $R_{dR}$ factors through $\ulMDM^{\eff}_{\gm}[T^{-1}]$ and $M$ is dualisable in $\ulMDM^{\eff}_{\gm}[T^{-1}]$. Then $R_{dR}(M)$ is finite dimensional. 
\end{cor}

The strategy of the proof goes in parallel with the one in \cite{KelMi1}. 
We first define presheaves $\ulMOmega^q$ on $\ulPSm_k$ and show that they are quasi-coherent \'etale sheaves (\S \ref{sec:MOmega}), admitting the action of correspondences (\S \ref{sec:MOmega-transfer}). 
Next we prove that their cohomology groups are cube invariant (\S \ref{sec:ciOmega}) and blow-up invariant (\S \ref{sec:biOmega}).
We also recall some basic terminologies and the construction of the modulus structure sheaf $\ulMO$ in \S \ref{sec:review-MO}.
In the final section \S \ref{sec:Hodge-Real}, we construct a modulus Hodge realisation. 
Moreover, we construct an extension of the de Rham realisation of Voevodsky's mixed motives, and prove that it is a monoidal functor.

\subsection*{Notation and Convention}

Once and for all we fix a field $k$. 
We write $\Sch = \Sch_k$ for the category of separated and of finite type schemes over $k$, and $\Sm = \Sm_k$ for the full subcategory consisting of smooth, separated, and of finite type schemes over $k$. 
We write $\Cor = \Cor_k$ for the category of finite correspondences over $k$ from \cite{VoeTri}.

\subsection*{Acknowledgements}
We thank Junnosuke Koizumi and Shuji Saito for pointing out that one can avoid the use of weak factorisation by using strong resolution of singularities, as explained in Remark \ref{rem:strong-res}.

\section{Review of $\ulMO$}\label{sec:review-MO}

In this section, we briefly recall the construction of the modulus structure sheaf $\ulMO$ from \cite{KelMi1} which we will use in the construction of the modulus Hodge sheaf. 
Throughout this paper, we will freely use the definitions and notations introduced in \cite{KelMi1}.
However, we recall some basic notions here for the reader's convenience. One can find a concise summary of the general theory in \cite{KelMi1}. For more detail, see \cite{kmsy1}, \cite{kmsy2}, \cite{kmsy3}, and \cite{KM21}.

\begin{defi}
A \emph{modulus pair} is a pair $\sX = (\hX,\mX)$ consisting of $\hX \in \Sch$ and an effective Cartier divisor $\mX \subset \hX$ such that $\iX := \hX \setminus |\mX| \in \Sm$. We call $\iX$ the \emph{interior} of $\sX$.
An \emph{ambient morphism} $(\hX,\mX) \to (\hY,\mY)$ is a morphism $f : \hX \to \hY$ in $\Sch$ with $f(\iX) \subset \iY$ and $\mX \geq f^* \mY$.
We write $\ulPSm = \ulPSm_k$ for the category consisting of modulus pairs and ambient morphisms. 
When we have $\mX = f^* \mY$, we say that the ambient morphism $f$ is \emph{minimal}.
\end{defi}

The idea of the theory of motives with modulus is to replace smooth schemes by modulus pairs, and $\AA^1$-homotopy invariance by $\bcube$-invariance, where $\bcube := (\PP^1,\infty)$ is a modulus pair that we call the \emph{cube}.
A very important property of the affine line $\AA^1$ is the existence of the multiplication map $\mu : \AA^1 \times \AA^1 \to \AA^1; (s,t) \mapsto st$. 
Unfortunately, one can only extend $\mu$ to a rational map $\PP^1 \times \PP^1 \dashrightarrow \PP^1$, not an entire morphism. 
This means that one do not have a multiplication map on $\bcube$ in the category $\ulPSm$. 
The above is one of the main reasons why we have to invert ``abstract admissible blow-ups'' in $\ulPSm$ to obtain a good category of modulus pairs $\ulMSm$.

\begin{defi}
An \emph{abstract admissible blow-up} is a morphism $f:(\hX,\mX) \allowbreak \to (\hY,\mY)$ in $\ulPSm$ with $\mX = f^*\mY$ such that $f : \hX \to \hY$ is proper, surjective, and an isomorphism over the interior $\iY \subset \hY$. 
Write $\Sigma$ for the class of abstract admissible blow-ups. 
We define $\ulMSm = \ulMSm_k := \ulPSm [\Sigma^{-1}]$.
\end{defi}

\begin{rema}
One can easily see that $\Sigma$ enjoys the calculus of right fraction, which implies that any morphism in $\ulMSm$ is represented by a diagram $\sX \xleftarrow{s} \sX' \xrightarrow{f} \sY$, where $s \in \Sigma$, and $f$ is an ambient morphism. 
\end{rema}

Thanks to the inversion of abstract admissible blow-ups, one can show that $\ulMSm$ admits categorical product $\times$, which is not the case for $\ulPSm$ (For the construction of the categorical product (more generally, fibre product), see \cite{kmsy1}, \cite{nistopmod}, \cite{KM21}).
However, we often use a different notion of product of modulus pairs. 

\begin{defi}
A \emph{tensor product} $\sX \otimes \sY$ of two modulus pairs $\sX$ and $\sY$ is defined by 
\[
\sX \otimes \sY := (\hX \times \hY, \mX \times \hY + \hX \times \mY).
\]
\end{defi}

\begin{rema}
The natural morphism $\sX \otimes \sY \to \sX \times \sY$ induced by the universal property of the categorical product is not an isomorphism in general. 
For example, the multiplication map $\bcube \otimes \bcube \to \bcube$ is well-defined in $\ulMSm$ but $\bcube \times \bcube \to \bcube$ is not.
\end{rema}

In the following, we frequently use the notion of normal crossing divisors. 
We introduce some convenient terminologies (cf. \cite[Definition A.1]{KelMi1}).

\begin{defi} \label{defi:NC}
Let $\sX$ be a modulus pair and $Z \subseteq \hX$ a closed subscheme. 
We will say that \emph{$Z$ has strict normal crossings with $\mX$} if for every point $x \in \hX$ the local ring $\OO_{\hX, x}$ is regular, and there exists a regular system of parameters $t_1, ..., t_n \in \OO_{\hX, x}$ such that
\[
\mX|_{\Spec(\OO_{\hX, x})} = \prod_{a \in A} t_a^{r_a}, \quad \textrm{ and } \quad Z|_{\Spec(\OO_{\hX, x})} = \Spec(\mathcal{O}_{\hX,x} / \langle t_b : b \in B \rangle)
\]
for some $r_a > 0$ and $A, B \subseteq \{1, ..., n\}$ (cf.\cite[00KU]{stacks-project}).

We will say that \emph{$Z$ has normal crossings with $\mX$} if there exists an \'etale covering $\hV \to \hX$ such that $Z \times_{\hX} \hV$ has strict normal crossings with $\mV$.
We say that $\sX$ is a \emph{normal crossings modulus pair} if $\varnothing$ has normal crossings with $\mX$.
\end{defi}

We briefly recall the sheaf of modulus global sections from \cite{KelMi1}.
For any ring $A$ and a nonzero divisor $f \in A$, set 
\[
\ulMO (A, f) := \{ a/f \in A[f^{-1}] : a \in \sqrt{(f)} \subseteq A \},
\]
where $\sqrt{(-)}$ denotes the radical of an ideal. We proved in \cite[Theorem 3.7]{KelMi1} that this local definition is extended to a global one by a standard patching argument:

\begin{theo} \label{theo:MOdefi}
There is a unique $\fppf$-sheaf $\ulMO$ on $\ulPSm_k$ such that for any modulus pair $\sX = (\hX,\mX)$ over $k$ with $\hX = \Spec A$ and $\mX = \Spec (A/fA)$ for some ring $A$ and a nonzero divisor $f \in A$, we have $\ulMO(\sX) = \ulMO(A, f)$. Furthermore, this is quasi-coherent as an {\'e}tale sheaf. In particular, its Zariski, Nisnevich, and {\'e}tale cohomologies agree, and vanish for affines.
\end{theo}

\begin{rema}
The sheaf we actually want to work with is $\sX \mapsto \ulMO(\sX^{\ic})$. Here, given a qcqs modulus pair $\sX = (\hX, \mX)$ (i.e., $\hX$ is a qcqs scheme, $\iX$ is an effective Cartier divisor) we write $\sX^{\ic} = ((\hX)^{\ic}, \mX|_{(\hX)^{\ic}})$ where $(\hX)^{\ic} = \underline{\Spec} (\OO_{\hX}^{\ic})$ and $\OO_{\hX}^{\ic}$ is the integral closure of $\OO_{\hX}$ in $\OO_{\iX}$.  More precisely, one can show that $\ulMO((-)^{\ic})$ is the $\ulMZar$-sheafification of $\ulMO$, cf.\cite[Lem.4.2, Rem.8.2]{KelMi1}. 
However, since we are basically interested in normal crossing modulus pairs in the present paper, the reader does not have to be bothered by this subtlety. 
\end{rema}

\section{The modulus Hodge sheaves}\label{sec:MOmega}

In this section, we construct the modulus Hodge sheaves $\ulMOmega^q$ on $\ulMSm$ by tensoring the modulus structure sheaf $\ulMO$ and the logarithmic Hodge sheaf. 
We also study some fundamental properties of them.

\subsection{The sheaf of modulus K\"ahler differentials $\protect\ulMOmega$}

Throughout this section, we continue with our fixed base field $k$.

\begin{defi}\label{def:PandMOmega}
Let $k \to A$ a $k$-algebra and $f \in A$ a nonzero divisor. 
We write $\ulPOmega^*_{(A,f)/k}$ for the smallest sub-dg-algebra of $\Omega^*_{A[f^{-1}]/k}$ containing $A=\Omega^0_{A/k}$ and $\dlog(a) := \frac{da}{a}$ for $a \in A \cap A[f^{-1}]^*$, and define
\[
\ulMOmega^q_{(A,f)/k} = \ulMO_{(A,f)} \cdot \ulPOmega^q_{(A,f)/k},
\]
still as a submodule of $\Omega^q_{A[f^{-1}]/k}$ (cf. Lemma~\ref{lem:MOmega-otimes}).
Note that $\ulPOmega_{(A,f)/k}^0=A$ and $\ulMOmega_{(A,f)/k}^0=\ulMO{(A, f)} \cdot A = \ulMO{(A, f)}$.
\end{defi}

\begin{remark}
For simplicity of notation, we often write
\[
\ulPOmega^*_{(A,f)} := \ulPOmega^*_{(A,f)/k}
\]
if there is no risk of confusion on the base.
\end{remark}

\begin{remark} \label{rema:sumPOmega}
By definition, elements of $\ulPOmega^q_{(A,f)}$ are sums of the form
\[ \sum_{i = 1}^N a_{i} \dlog(b_{i,1}) \wedge \dots \wedge \dlog(b_{i,n_i}) \wedge d(c_{i,1}) \wedge \dots \wedge d(c_{i, m_i}) \]
where $a_i, c_{i,j} \in A, b_{i,j} \in A \cap A[f^{-1}]^*$ and $n_i + m_i = q$. Elements of $\ulMOmega^q_{(A,f)}$ are the same but we allow $a_i \in \ulMO_{(A,f)}$.
\end{remark}

\begin{exam}\label{ex:locfreeomega}
If $A = \QQ[x_1, \dots, x_n]$ regarded as a $\mathbb{Q}$-algebra and $f = x_1^{r_1}\dots x_i^{r_i}$ with $r_1, \dots, r_i > 0$, then 
\begin{align*}
\ulPOmega_{(A, f)}^1 
&= 
\left ( \bigoplus_{j = 1}^i \QQ[x_1, \dots, x_n] \cdot \tfrac{dx_j}{x_j} \right ) 
\oplus 
\left ( \bigoplus_{j = i+1}^n \QQ[x_1, \dots, x_n] \cdot dx_j \right ) \\
&= \Omega^1_{\hX}(\log \mX)
\end{align*}and 
\begin{align*}
\ulMOmega_{(A, f)}^1 
&= 
\left ( \bigoplus_{j = 1}^i \QQ[x_1, \dots, x_n] \cdot \tfrac{1}{x_1^{r_1{-}1}\dots x_i^{r_i{-}1}} \cdot \tfrac{dx_j}{x_j} \right ) 
\\&\oplus 
\left ( \bigoplus_{j = i+1}^n \QQ[x_1, \dots, x_n] \cdot \tfrac{1}{x_1^{r_1{-}1}\dots x_i^{r_i{-}1}} \cdot dx_j \right ) \\
&= \Omega^1_{\hX}(\log \mX)(\mX {-} |\mX|)
\end{align*}
and for general $q \geq 0$ 
\[ \ulPOmega_{(A, f)}^q = \Omega^q_{\hX}(\log \mX) \]
and
\[ \ulMOmega_{(A, f)}^q = \Omega^q_{\hX}(\log \mX)(\mX {-} |\mX|) \]
where $(\hX, \mX) = \Spec(\QQ[x_1, \dots, x_n], x_1^{r_1}\dots x_i^{r_i})$.
cf.\cite[Cor.6.8]{RS18}, \cite{IY17}.
\end{exam}

\begin{exam} \label{exam:sncOmegaFree}
Similarly, if $A$ is a localisation of a smooth $k$-algebra at a maximal ideal $\m$ (so in particular, $A$ is a regular local ring, \cite[07R9,038T]{stacks-project}), we choose a regular system of parameters $x_1, \dots, x_n \in \m$, \cite[12.J]{Mat80}, and set $f = x_1^{r_1}\dots x_i^{r_i}$ with $r_1, \dots, r_i > 0$ then we have similar expressions. 

\begin{align*}
\ulPOmega_{(A, f)}^1 
&= 
\left ( \bigoplus_{j = 1}^i A \cdot \tfrac{dx_j}{x_j} \right ) 
\oplus 
\left ( \bigoplus_{j = i+1}^n A \cdot dx_j \right ) \\
&= \Omega^1_{\hX}(\log \mX)
\end{align*}and 
\begin{align*}
\ulMOmega_{(A, f)}^1 
&= 
\left ( \bigoplus_{j = 1}^i A \cdot \tfrac{1}{x_1^{r_1{-}1}\dots x_i^{r_i{-}1}} \cdot \tfrac{dx_j}{x_j} \right ) 
\oplus 
\left ( \bigoplus_{j = i+1}^n A \cdot \tfrac{1}{x_1^{r_1{-}1}\dots x_i^{r_i{-}1}} \cdot dx_j \right ) \\
&= \Omega^1_{\hX}(\log \mX)(\mX {-} |\mX|)
\end{align*}
and for general $q \geq 0$ 
\[ \ulPOmega_{(A, f)}^q = \Omega^q_{\hX}(\log \mX) \]
and
\[ \ulMOmega_{(A, f)}^q = \Omega^q_{\hX}(\log \mX)(\mX {-} |\mX|) \]
where $(\hX, \mX) = \Spec(A, x_1^{r_1}\dots x_i^{r_i})$.
Indeed, by hypothesis $\Omega^1_A$ and $\Omega^1_{A[f^{-1}]}$ are free with basis $dx_1, \dots, dx_n$,
so it suffices to check that $A \cap A[f^{-1}]^* = A^* \times x_1^{\ZZ} \times \dots \times x_i^{\ZZ}$. This follows from the fact that $A$ (and $A[f^{-1}]$) is a UFD \cite[0AG0]{stacks-project} and the prime factors of $f$ are $x_1, \dots, x_i$.

In particular, suppose that $\hX$ is smooth and affine, and suppose $\mX = \sum_{i = 0}^j n_i D_i$ with $n_i > 0$ and $\mX$ has strict normal crossings support. Then there are short exact sequences of $\OO_{\hX}$-modules:
\begin{align}
\ulPOmega^q_{(\hX, \mX{-}D_0)} &= \ulPOmega^q_{(\hX, \mX)} ,\qquad 
 & n_0 > 1 \label{equa:POmegaSeS1} \\
0 \to \ulPOmega^q_{(\hX, \mX{-}D_0)} &\to \ulPOmega^q_{(\hX, \mX)} \to 
i_*\Omega^{q-1}_{D_0} 
\to 0 ,
 & n_0 = 1  \label{equa:POmegaSeS2} 
\end{align}
and
\begin{align}
0 \to \ulMOmega^q_{(\hX, \mX{-}D_0)} &\to \ulMOmega^q_{(\hX, \mX)} \to i_*i^*\ulPOmega^q_{(\hX, \mX)} \to 0,
 & n_0 > 1  \label{equa:MOmegaSeS1} \\
0 \to \ulMOmega^q_{(\hX, \mX{-}D_0)} &\to \ulMOmega^q_{(\hX, \mX)} \to 
\ulMO_{\mX} \otimes i_*\Omega^{q-1}_{D_0} 
\to 0,
 & n_0 = 1 \label{equa:MOmegaSeS2}
\end{align}
where $i:D_0 \to X$ is the inclusion.
\end{exam}

\begin{rema} \label{rema:ulMZarification}
Ideally, we would like to reserve the symbol $\ulMOmega^*$ for the $\ulMZar$-sheafification of the $\ulMOmega^*$ defined above. However, in the presence of resolution of singularities and weak factorisation, for modulus pairs with regular total space and modulus with normal crossing support the above definition already captures the $\ulMZar$-sheafification. So we just work with that in this article.
\end{rema}

\begin{lemm}[Functoriality] \label{lemm:ulPOmegaFunct}
Let $\phi: A \to B$ be a $k$-algebra homomorphism, and $f \in A$, $g \in B$ nonzero divisors. 
Assume that $\phi(f)$ divides $g$ in $B$.
Then the canonical morphism $A[f^{-1}] \to B[g^{-1}]$ induces a morphism of graded $A$-modules
\[ \ulPOmega^*_{(A,f)} \to \ulPOmega^*_{(B,g)}. \] 
\end{lemm}

\begin{proof}
Follows directly from the definition and the functoriality of $A \cap A[f^{-1}]^*$.
\end{proof}

\begin{lemm}[Filtered colimits] \label{lemmulPOmegaColim}
If $(A_\lambda)_{\lambda \geq 0}$ is a filtered system of $k$-algebras, $f_0 \in A_0$ is a nonzero divisor such that the images $f_\lambda \in A_\lambda$ are also nonzero divisors, then
\[ \varinjlim \ulPOmega^*_{(A_\lambda,f_\lambda)} \stackrel{\cong}{\to} \ulPOmega^*_{(A,f)}. \]
where $A = \varinjlim A_\lambda$ and $f$ is the image of the $f_0$ in $A$.
\end{lemm}

\begin{proof}
The canonical comparison morphism is compatible with the canonical inclusions into $\varinjlim \Omega^*_{A_\lambda[\frac{1}{f_\lambda}]} \cong \Omega^*_{A[\frac{1}{f}]}$, so it is injective. For surjectivity, it suffices to show that for any $a \in A \cap (A[\frac{1}{f}]^*)$, the element $\dlog(a)$ is in the image. But any such $a$ can be lifted to some $A_\lambda \cap (A_\lambda[f_\lambda^{-1}]^*)$.
Indeed, we can lift to some $b \in A_\lambda$. The condition $b \in A_\lambda[f_\lambda^{-1}]^*$ is equivalent to the existence of some $c \in A_\lambda$ with $bc = f^n$ for some $n$. Such a $c$ exists in the limit since $a \in A \cap (A[\frac{1}{f}]^*)$, so up to changing $\lambda$, we can find one in $A_\lambda$.
\end{proof}

\begin{prop} \label{prop:POmegaLoc}
Let $A$ be a $k$-algebra and $f, g \in A$ nonzero divisors. The canonical comparison morphism of dg-$\Omega^*_A[\frac{1}{g}]$-modules 
\begin{equation} \label{equa:POmegaLoc}
\ulPOmega^*_{(A,f)}[\tfrac{1}{g}] \to \ulPOmega^*_{(A[\frac{1}{g}],f)}
\end{equation}
is surjective in the following two cases:
\begin{enumerate}
 \item \label{prop:POmegaLoc:supp} $A[\tfrac{1}{f}] = A[\frac{1}{fg}]$ (that is, $supp(g) \subseteq supp(f)$).
 \item \label{prop:POmegaLoc:PID} $A$ is a UFD.
\end{enumerate}
and injective in the case
\begin{enumerate} \setcounter{enumi}{2}
 \item \label{prop:POmegaLoc:TF} $\Omega^1_{A[\frac{1}{f}]}$ is $g$-torsion free.
\end{enumerate} 
\end{prop}

\begin{proof}
We show surjectivity of Eq.\eqref{equa:POmegaLoc}. Since $\Omega^q_A[\frac{1}{g}] \cong \Omega^q_{A[\frac{1}{g}]}$, for surjectivity it suffices to show that $\dlog(\frac{a}{g^n})$ is in the image of Eq.\eqref{equa:POmegaLoc} for any $\frac{a}{g^n} \in A[\frac{1}{g}] \cap A[\frac{1}{fg}]^*$, cf.Rem.\ref{rema:sumPOmega}. 

First note that we can assume $n = 0$. Indeed, we always have $g \in A[\frac{1}{g}] \cap A[\frac{1}{fg}]^*$ and $\dlog \frac{a}{g^n} = \dlog a - n \dlog g$ (and of course $\frac{a}{g^n}$ is a unit if and only if $a = \frac{a}{g^n}g^n$ is a unit).

 Therefore, it suffices to show that $\dlog(a)$ is in the image of Eq.\eqref{equa:POmegaLoc} for any $a \in A \cap A[\frac{1}{fg}]^*$. In case the case \eqref{prop:POmegaLoc:supp} this follows from $A \cap A[\tfrac{1}{f}]^* = A \cap A[\tfrac{1}{fg}]^*$. Consider the case \eqref{prop:POmegaLoc:PID}. Since $a \in A[\tfrac{1}{fg}]^*$ there is $b \in A$ such that $ab = (fg)^m$ for some $m \in \NN$. 
For any factorisation $g = \pi \theta$ with $\pi, \theta \in A$, the element $\dlog(\pi) = \tfrac{\theta}{g} d \pi$ is in the image of Eq.\eqref{equa:POmegaLoc}. 
On the other hand, for any factoriation $f = \pi \theta$, we have both $\pi, \theta \in A \cap A[\frac{1}{f}]^*$ so $\dlog(\pi)$ is in $\ulPOmega^1_{(A,f)}$ and therefore in the image of Eq.\eqref{equa:POmegaLoc}. 
It follows that $\dlog(\pi)$ is in the image of Eq.\eqref{equa:POmegaLoc} for any irreducible factor $\pi$ of $(fg)^m = ab$, and hence for any irreducible factor of $a$.
Therefore, $\dlog(a)$ is in the image.

Finally, consider case \eqref{prop:POmegaLoc:TF}. To prove injectivity, it suffices to show that both sides inject into $\Omega^1_{A[\frac{1}{fg}]}$. For the target this is by definition. For the source, choose an element $\frac{\omega}{g^n}$ in the kernel, where $\omega \in \ulPOmega^1_{(A,f)}$. Now $\ulPOmega^1_{(A, f)} \to \Omega^1_{A[\frac{1}{f}]}$ is injective by definition and $\Omega^1_{A[\frac{1}{f}]} \to \Omega^1_{A[\frac{1}{gf}]}$ is injective by the hypothesis Eq.\eqref{prop:POmegaLoc:TF}. Hence $\frac{\omega}{g^n} = 0$ implies $g^n \frac{\omega}{g^n} = \omega = 0$.
\end{proof}

\begin{rema}
One can show that if $A = k[a, b, f, g] / (ab-gf)$ then $\ulPOmega^1_{(A,f)}[\frac{1}{g}] \to \ulPOmega^1_{(A[\frac{1}{g}],f)}$ is not surjective. So the UFD hypothesis in Proposition~\ref{prop:POmegaLoc} is well-justified.
\end{rema}

\begin{cor}
If $\hX$ is a locally Noetherian, locally UFD scheme (e.g., regular) and $\mX$ an effective Cartier divisor, there exists a unique quasi-coherent sheaf $\ulPOmega^q_{(\hX, \mX)}$ on the small Zariski site $\hX_{\Zar}$ such that for open affines $U = \Spec(A) \subseteq \hX$ such that $U \times_{\hX} \mX$ is globally principal, say $(f)$, and we have
\[ \ulPOmega^q_{(\hX, \mX)}(U) = \ulPOmega^q_{(A, f)}. \]
Moreover, in this case the stalk at $x \in \hX$ is give by 
\[ (\ulPOmega^q_{(\hX, \mX)})_x = \ulPOmega^q_{(\OO_{X, x}, f)} \]
where, again $\Spec(\OO_{X, x}) \times_{\hX} \mX = (f)$.
\end{cor}

\begin{proof}
The claim follows from the fact that regular locally Noetherian rings are UFD's satisfying $\Omega^1_{A_\lambda} \subseteq \Omega^1_{\Frac(A_\lambda)}$ on every connected component $A_\lambda$. The localisation claim  comes from Lemma~\ref{lemmulPOmegaColim}.
\end{proof}

\begin{lemma}\label{lem:extPOmega}
For any $k$-algebra $A$ and a non-zero divisor $f \in A$, consider the morphism of dg-modules
\begin{equation}\label{eqOmega1}
\bigwedge\nolimits_A^q \ulPOmega^1_{(A,f)} \to \ulPOmega^q_{(A,f)} 
\end{equation}
which is induced by the universal property of exterior product.
Then Eq.\eqref{eqOmega1} is surjective. 

Moreover, suppose that $\ulPOmega^1_{(A,f)}$ is a flat $A$-module. Then Eq.\eqref{eqOmega1} is an  isomorphism. In particular, both sides of $Eq.\eqref{eqOmega1}$ are flat over $A$.
\end{lemma}

\begin{proof}
This map Eq.\eqref{eqOmega1} is surjective by the minimality of $\ulPOmega^*_{(A,f)}$ (see also Rem.~\ref{rema:sumPOmega}).
To prove the second assertion, it suffices to show that the composition \[
\bigwedge\nolimits_A^q \ulPOmega^1_{(A,f)} \to \ulPOmega^q_{(A,f)} \to \Omega^q_{A[f^{-1}]}= \bigwedge\nolimits_A^q \Omega^1_{A[f^{-1}]}
\]
 is injective. 
Consider the following commutative diagram:
\[\xymatrix{
\bigwedge\nolimits_A^q \ulPOmega^1_{(A,f)} \ar[r]^a \ar[d]_b & \bigwedge\nolimits_A^q \Omega^1_{A[f^{-1}]} \ar[r]^{\sim} \ar[d] & \Omega^q_{A[f^{-1}]} \ar@{=}[d] \\
\left( \bigwedge\nolimits_A^q \ulPOmega^1_{(A,f)} \right) \otimes_A A[f^{-1}] \ar[r]^{\sim} &\bigwedge\nolimits_A^q  \left( \ulPOmega^1_{(A,f)} \otimes_A A[f^{-1}]  \right) \ar[r]^(0.68)c & \Omega^q_{A[f^{-1}]}
}\]
We claim that $b$ and $c$ are injective, and hence so is $a$.

First we prove that $b$ is injective. To see this, it suffices to prove that $\bigwedge\nolimits_A^q \ulPOmega^1_{(A,f)}$ is $A$-flat. 
This follows from the flatness of $\ulPOmega^1_{(A,f)}$ since any exterior power of any flat $A$-module is $A$-flat. 
We can see this general fact as follows: when $M$ is a free $A$-module of finite rank, then one shows that its exterior powers are free (and hence flat) by a direct calculation. Noting that exterior product commutes with direct limits, one can prove the general case by using Lazard's theorem which asserts that an $A$-module is flat if and only if it is a direct limit of free $A$-modules of finite rank.

Finally, to see that $c$ is injective, it suffices to apply Prop.~\ref{prop:POmegaLoc} (3). In fact, $c$ is also surjective by Prop.~\ref{prop:POmegaLoc} (1) since $\Omega^1_{A[f^{-1}]}$ is an $A[f^{-1}]$-module.
\end{proof}

\begin{lemma}\label{lem:MOmega-otimes}
For any $k$-algebra $A$ and a non-zero divisor $f \in A$ such that $\ulPOmega^1_{(A,f)}$ is a flat $A$-module, the canonical surjection
\[
\ulMO_{(A,f)} \otimes_A \ulPOmega^q_{(A,f)} \twoheadrightarrow \ulMOmega^q_{(A,f)} 
\]
is an isomorphism of $A$-modules for any $q \geq 0$.
\end{lemma}

\begin{proof}
We have $\ulMOmega^q_{(A,f)} \subseteq \Omega^q_{A[f^{-1}]}$ by definition so it suffices to show that the map
\[
\ulMO_{(A,f)} \otimes_A \ulPOmega^q_{(A,f)} 
\xrightarrow{a}
A[\tfrac{1}{f}] \otimes_A \ulPOmega^q_{(A,f)}
\xrightarrow{b}
A[\tfrac{1}{f}] \otimes_A \Omega^q_{A[f^{-1}]}
 \cong \Omega^q_{A[f^{-1}]}
\]
is injective, where $a$ is induced by the inclusion $\ulMO_{(A,f)} \subset A[\tfrac{1}{f}]$, and $b$ is induced by the inclusion $\ulPOmega^q_{(A,f)} \subset \Omega^q_{A[f^{-1}]}$.
The map $a$ is injective since $\ulPOmega^q_{(A,f)}$ is a flat $A$-module
 by Lemma~\ref{lem:extPOmega}.
The map $b$ is injective since $A[f^{-1}]$ is flat over $A$.
\end{proof}

\begin{cor}\label{cor:qcoh-MOmega}
Let $\hX$ be a regular locally Noetherian scheme over $k$ and $\mX$ an effective Cartier divisor. 
Then there exists a unique quasi-coherent sheaf $\ulMOmega^q_{(\hX, \mX)}$ on the small Zariski site $\hX_{\Zar}$ such that for open affines $U = \Spec(A) \subseteq \hX$ such that $U \times_{\hX} \mX$ is globally principal, say $(f)$, we have
\[ \ulMOmega^q_{(\hX, \mX)}(U) = \ulMOmega^q_{(A, f)}. \]
Moreover, the stalk at $x \in \hX$ is given by 
\[ (\ulMOmega^q_{(\hX, \mX)})_x = \ulMOmega^q_{(\OO_{X, x}, f)} \]
where, again $\Spec(\OO_{X, x}) \times_{\hX} \mX = (f)$.
\end{cor}

\begin{proof}
Define $\ulMOmega^q_{(\hX, \mX)}$ to be the quasi-coherent $\OO_{\hX}$-module
\[
\ulMOmega^q_{(\hX, \mX)} := \ulMO_{(\hX, \mX)} \otimes_{\OO_{\hX}} \ulPOmega^q_{(\hX, \mX)},
\]
where $\ulPOmega^q_{(\hX, \mX)}$ is the quasi-coherent sheaf from Proposition~\ref{prop:POmegaLoc}.
Then this sheaf satisfies the desired property by Lemma~\ref{lem:MOmega-otimes}.
\end{proof}

\subsection{\'Etale descent for $\protect \ulMOmega^q$}

In this subsection we prove the \'etale version of Corollary~\ref{cor:qcoh-MOmega}.

\begin{lemm} \label{lemm:Omegaheightone}
Suppose that $A$ is a smooth $k$-algebra, and $f$ is a strict normal crossings divisor. Then for each $q \geq 0$ we have 
\begin{align*}
\ulPOmega^q_{(A, f)} &= \bigcap_{\substack{\p \in \Spec(A) \\ ht(\p) = 1}}  \ulPOmega^q_{(A_\p, f)} \\
\ulMOmega^q_{(A, f)} &= \bigcap_{\substack{\p \in \Spec(A) \\ ht(\p) = 1}} \ulMOmega^q_{(A_\p, f)} 
\end{align*}
where the intersection takes place in $\Omega^q_{\Frac(A)}$.
\end{lemm}

\begin{proof}
In general, for any locally free $A$-module $M$ of finite rank over an integral $A$, we have $M \stackrel{\sim}{\to} \lim_{\p \in \Spec(A)} M_\p$. Indeed, for finite free $M$, it follows from the case $M = A$. For a general finite rank locally free module, it follows by Zariski descent. Since $M$ is torsion-free it is identified with a submodule of $M \otimes_A \Frac(A)$ so the limit is an intersection.

Therefore, by Proposition~\ref{prop:POmegaLoc}, resp. Corollary~\ref{cor:qcoh-MOmega}, we have
 
\[
\ulPOmega^q_{(A, f)} = \bigcap_{\p \in \Spec(A)} \ulPOmega^q_{(A_\p, f)}, \quad 
\ulMOmega^q_{(A, f)} = \bigcap_{\p \in \Spec(A)} \ulMOmega^q_{(A_\p, f)},
\]
and hence we may replace $A$ with the localisation at a maximal ideal $\m$. In this case, $\ulPOmega^q_{(A, f)}$ (resp. $\ulMOmega^q_{(A, f)}$) is globally free of finite rank by Example~\ref{exam:sncOmegaFree}. Therefore, the result follows from the corresponding statement for $A$, i.e., \[A = \bigcap_{\substack{\p \in \Spec(A) \\ ht(\p) = 1}} A_\p.\] 
\end{proof}

\begin{lemm} \label{lemm:dvretaleOmega}
Suppose that $A \to B$ is an \'etale morphism and $A$ is a dvr. 
Then 
\begin{align*}
\ulPOmega^q_{(B, f)} &= B \otimes_A \ulPOmega^q_{(A, f)}, \\
\ulMOmega^q_{(B, f)} &= B \otimes_A \ulMOmega^q_{(A, f)}.
\end{align*}
Consequently, these are quasi-coherent \'etale sheaves on the small \'etale site on $\Spec(A)$.
\end{lemm}

\begin{proof}
Consider the square
\begin{equation}
 \xymatrix{
B \otimes_A \ulPOmega^q_{(A, f)} \ar[r]^-\phi \ar[d]_\psi & \ulPOmega^q_{(B, f)} \ar[d] \\
B \otimes_A \Omega^q_{A[f^{-1}]} \ar[r]_-\theta & \Omega^q_{B[f^{-1}]} 
} 
\end{equation}
The canonical morphism $\phi$ is injective because $\psi$ and $\theta$ are. Indeed, $A \to B$ is flat and $\ulPOmega^q_{(A, f)} \to \Omega^q_{A[f^{-1}]}$ is injective so $\psi$ is injective, and $\theta$ is an isomorphism because $A \to B$ is \'etale.

For surjectivity, since $A$ is a dvr, $B$ is semi-local with dvr local rings (at maximal ideals), so by Proposition~\ref{prop:POmegaLoc} (resp. Corollary~\ref{cor:qcoh-MOmega}) applied to $(B, f)$ we can assume that $B$ is also a dvr.

To show that $\phi$ is surjective, it suffices to show that for every $b \in B \cap B[f^{-1}]^*$ the element $\dlog(b)$ is in the image of $\phi$. Since $B$ is a dvr, we have $B[f^{-1}] = \Frac(B)$, so $ B \cap B[f^{-1}]^* = B \setminus \{0\}$. Since $A \to B$ is a \'etale, $\m_B = \m_A B$, and since both are dvrs, any uniformiser $\pi$ for $A$ is also a uniformiser for $B$. 
Now since $B$ is a dvr, every element $b \in B$ is of the form $b = u\pi^n$ for some $u \in B^*$ where $\pi \in A$ is our chosen uniformiser. Then $\dlog(b) = \dlog(u) + \dlog(\pi^n) = u^{-1}du + \dlog(\pi^n)$ with $\pi \in A \setminus \{0\} = A \cap A[f^{-1}]^*$. %
Clearly, $\dlog(\pi^n)$ is in the image of $\phi$. For $u^{-1}du$, note $u^{-1} \in B$ and $B\otimes_A \Omega^q_A = \Omega^q_B$ because $A \to B$ is \'etale, so $u^{-1}du$ is also in the image of $\phi$.
\end{proof}

\begin{prop} \label{prop:OmegaEtaleQC}
Suppose that $A \to B$ is an \'etale morphism between smooth $k$-algebras and $f \in A$ is strict normal crossings. 
Then the canonical morphisms
\begin{align*}
B \otimes_A \ulPOmega^q_{(A, f)} &\stackrel{\sim}{\to} \ulPOmega^q_{(B, f)} \\
B \otimes_A \ulMOmega^q_{(A, f)} &\stackrel{\sim}{\to} \ulMOmega^q_{(B, f)}
\end{align*}
are isomorphisms. In other words, both $\ulPOmega^q_{(A, f)}$ and $\ulMOmega^q_{(A, f)}$ are quasi-coherent \'etale sheaves on the small \'etale site $\Spec(A)_{\et}$.
\end{prop}

\begin{proof}
First we prove the claim for $\ulPOmega$.
\begin{align*}
\ulPOmega^q_{(B, f)} 
&\stackrel{\textrm{Lem.\ref{lemm:Omegaheightone}}}{=} 
\bigcap_{\substack{\p \in \Spec(B) \\ ht(\p) = 1}} \ulPOmega^q_{(B_\p, f)} 
\\&\stackrel{\textrm{Lem.\ref{lemm:dvretaleOmega}}}{=} 
\bigcap_{\substack{\p \in \Spec(B) \\ ht(\p) = 1}} B_\p \otimes_{A_\q} \ulPOmega^q_{(A_\q, f)} 
\\&\stackrel{\textrm{Prop.\ref{prop:POmegaLoc}}}{=} 
\bigcap_{\substack{\p \in \Spec(B) \\ ht(\p) = 1}} B_\p \otimes_{A_\q} A_\q \otimes_A \ulPOmega^q_{(A, f)} 
\\&\stackrel{\textrm{}}{=}
\bigcap_{\substack{\p \in \Spec(B) \\ ht(\p) = 1}} B_\p  \otimes_A \ulPOmega^q_{(A, f)} 
\\&\stackrel{}{=}
 B \otimes_A \ulPOmega^q_{(A, f)} ,
\end{align*}
where the last one is because $\ulPOmega^q_{(A, f)} \cong A^{\oplus N}$ for some $N$, by Example~\ref{exam:sncOmegaFree}.

The claim for $\ulMOmega$ follows from that for $\ulPOmega$. Indeed, since $\ulPOmega^q_{(A,f)}$ (resp. $\ulPOmega^q_{(B,f)}$) is a flat (or even free) $A$-module (resp. $B$-module) by Example~\ref{exam:sncOmegaFree}, we have $\ulMOmega^q_{(A,f)} = \ulMO_{(A,f)} \otimes_{A} \ulPOmega^q_{(A,f)}$ and $\ulMOmega^q_{(B,f)} = \ulMO_{(B,f)} \otimes_{B} \ulPOmega^q_{(B,f)}$. 
Moreover, we have $\sqrt{f \OO_A} \cdot \OO_B = \sqrt{f \OO_B}$ by the \'etaleness of $A \to B$. 
Thus, we have 
\begin{align*}
B \otimes_A \ulMOmega^q_{(A,f)} 
&\cong B \otimes_A \ulMO_{(A,f)} \otimes_A \ulPOmega^q_{(A,f)}\\
&\cong (B 
\otimes_A \ulMO_{(A,f)})  \otimes_B (B \otimes_A \ulPOmega^q_{(A,f)}) \\
&\cong \ulMO_{(B,f)} \otimes_B \ulPOmega^q_{(B,f)} \\
&\cong \ulMOmega^q_{(B,f)},
\end{align*}
as desired.
\end{proof}

\section{Transfers on $\ulMOmega$}\label{sec:MOmega-transfer}

In the previous section, we constructed the modulus Hodge sheaves $\ulMOmega^q$ as sheaves on the category $\ulMSm_k$.  
In this section, assuming that the field $k$ has characteristic $0$, we will show that $\ulMOmega^q$ admits ``transfers'', that is, $\ulMOmega^q$ can be extended to a sheaf on the category of modulus correspondences $\ulMCor$.
First we briefly recall what $\ulMCor$ is. 

\begin{defi}
An \emph{elementary modulus correspondence} $\sX \to \sY$ between two modulus pairs $\sX$ and $\sY$ is an elementary finite correspondence $V \in \Cor (\iX,\iY)$ satisfying the following property: let $\hV \subset \hX \times \hY$ be the closure, and $\hV^N \to \hV$ be its normalisation. Let $p:\hV^N \to \hV \to \hX \times \hY \to \hX$ be the natural composition (define $q : \hV^N \to \hY$ similarly).  Then we have 
\begin{enumerate}
\item $p$ is proper,
\item $p^*\mX \geq q^*\mY$. 
\end{enumerate} 

\begin{rema}
If we allow ourselves to use the general notion of modulus pairs from \cite{KM21}, where the interior of a modulus pair is not assumed to be smooth, then the condition (2) above is equivalent to that $q$ induces a morphism of modulus pairs $(\ol{V}^N, p^*\mX) \to \sY$. 
\end{rema}

A \emph{modulus correspondence} $\sX \to \sY$ is a $\Z$-linear finite formal sum of elementary modulus correspondences from $\sX$ to $\sY$. 
Write $\ulMCor (\sX,\sY)$ for the group of modulus correspondences from $\sX$ to $\sY$. Then we have $\ulMCor (\sX,\sY) \subset \Cor (\iX,\iY)$, and one can show that the composition of finite correspondences induce that of modulus correspondences. 
Thus, we obtain the category $\ulMCor = \ulMCor_k$ consisting of modulus pairs and modulus correspondences. 
Moreover, one can easily check that the graph functor $\Sm \to \Cor ; f \mapsto \Gamma_f$ restricts to $\ulMSm \to \ulMCor$ (see \cite{kmsy1} for the detailed proofs of these facts).
\end{defi}

Recall from \cite{LcW09} that the Hodge sheaf $\Omega^q$ admits the structure of a presheaf with transfers.
That is, it extends to an additive presheaf on $\Cor$.
The following theorem shows that the transfers on $\Omega^q$ is inherited by $\ulMOmega^q$, at least in characteristic $0$:

\begin{thm}\label{thm:transfer-MOmega}
Suppose $\mathrm{char} (k)=0$.
Let $\sX,\sY \in \ulMSm$ with $\sX$ normal crossing and $\hY$ normal. 
Then, for any modulus correspondence $\ulMCor (\sX,\sY) \subset \Cor (\iX,\iY)$ and for any $q \geq 0$, there exists a unique map $\ulMOmega^q (\sY) \to \ulMOmega^q (\sX)$ which makes the following diagram commute:
\[\xymatrix{
\ulMOmega^q (\sY) \ar[r] \ar[d] & \ulMOmega^q (\sX) \ar[d] \\
\Omega^q (\iY) \ar[r]^{\alpha^*} & \Omega^q (\iX) ,
}\] 
where the vertical maps are natural inclusions, and the horizontal bottom map is induced by the transfer structure on $\Omega^q$.
\end{thm}

\begin{rema}
When $\sX$ and $\sY$ are strict normal crossing and $\hX$ and $\hY$ are proper over a field $k$ of characteristic $0$, then the above theorem is a direct consequence of the result of R\"ulling-Saito, \cite[Cor.~6.8]{RS18}. 
Indeed, under these conditions, we have $\tilde{\Omega}^q (\sX) = \ulMOmega^q (\sX)$, where $\tilde{\Omega}^q := \omega^{\CI} \Omega^q$ is a cube invariant sheaf on $\ulMCor$ defined as ``the largest cube invariant model of $\Omega^q$'' (See \cite[Proposition 3.26]{BKS22} for the definition of $\omega^{\CI}$). Since the left hand side has an action of modulus correspondences, the right hand side also does. And we can also drop the properness condition by using the fact that for any modulus pair $\sX$ and for any proper compactification $\hX \subset \hZ$ such that $\mX$ extends to an effective Cartier divisor $X^{\infty}_1$ and $\hZ \setminus \hX = |D|$ for an effective Cartier divisor $D \subset \hZ$, one has $\ulMOmega^q(\hX) = \colim_{n \geq 1} \ulMOmega^q (\hZ,X^{\infty}_1+nD)$.

Although this case is enough for our main purpose, we will provide a different proof of Theorem \ref{thm:transfer-MOmega} for the sake of completeness and a bit more generality.
\end{rema}

\begin{proof}
By \cite[Recollection 2.1]{KelMi1}, there is a morphism of modulus pairs $\sW \to \sX$ such that $\hW$ is normal, integral, $\hW \to \hX$ is proper surjective, and the composition $\sW \to \sX \to \sY$ is a finite sum of morphisms of modulus pairs. 
Normalising, we can assume $\hW$ is integrally closed in $\iW$. 
As such, the morphism $(\ast)$ in the diagram
\[ \xymatrix{
\Omega^q(\iY) \ar[r] &\Omega^q(\iX) \ar[r]^{\subseteq} & \Omega^q(\iW) \\
\ulMOmega^q(\sY) \ar@{-->}[r]^{(\ast\ast)} \ar@/_12pt/@{-->}[rr]_{(\ast)} \ar[u]_{\cup|} &\ulMOmega^q(\sX) \ar[r] \ar[u]_{\cup|} & \ulMOmega^q(\sW) \ar[u]_{\cup|}
} \]
certainly exists, and is unique by injectivity of $\ulMO(\sW) \subseteq \OO(\iW)$. By Lem.~\ref{lem:MOmegaCart} below, the square on the right is Cartesian, so the morphism $(\ast\ast)$ also exists and is unique.
\end{proof}

By the above argument, Theorem \ref{thm:transfer-MOmega} is reduced to the following lemma.

\begin{lemm}\label{lem:MOmegaCart}
Suppose $\mathrm{char}k = 0$.
Let $p : \sY \to \sX$ be an ambient morphism with $\hY$ normal and $\sX$ normal crossing such that $p : \hY \to \hX$ is generically \'etale, dominant, surjective morphism such that $\mY = p^*\mX$. 
Then, for any $q \geq 0$, the square 
\[\xymatrix{
\ulMOmega^q (\sX) \ar[r] \ar[d] & \ulMOmega^q (\sY) \ar[d] \\
\Omega^q (\iX) \ar[r] & \Omega^q (\iY) 
}\]
is Cartesian. 
In particular, there exists a unique presheaf on $\ulMSm$ whose restriction to $\ulPSm$ agrees with $\ulMOmega^q$ on integrally closed modulus pairs. 
\end{lemm}

\begin{proof}
The case $q=0$ is known by \cite[Proposition 4.3]{KelMi1}. We consider the case $q \geq 1$.
Notice that all objects in the above diagram are subobjects of $\Omega^q (\iY)$. 
By Lem.~\ref{lemm:Omegaheightone} and by the assumption that $\sX$ is normal crossing, it suffices to show that for any codimension one point $x \in \hX$ and for any point $y \in \hY$ lying over $x$, the square 
\[\xymatrix{
\ulMOmega^q (\sX_x) \ar[r] \ar[d] & \ulMOmega^q (\sY_x) \ar[d] \\
\Omega^q (\iX_x) \ar[r] & \Omega^q (\iY_y) 
}\]
is Cartesian, where $\sX_x := (\Spec \OO_{X,x}, \mX |_{\Spec \OO_{X,x}})$ (and similarly for $\sY_y$). Note that $y \in \hY$ is also of codimension one, and hence $\hY_y$ is a spectrum of a dvr since $\hY$ is normal by assumption.
By replacing $\sX$ and $\sY$ by $\sX_x$ and $\sY_y$ respectively, we may and do assume from the beginning that $\hX$ and $\hY$ are the spectra of dvr's. 

Since $p$ is generically \'etale by assumption, we have $\iX = \Spec K$ and $\iY = \Spec L$ for a finite separable field extension $L/K$ of dvf's. Then we have $\OO_{\hX} = \OO_K$ and $\OO_{\hY} = \OO_L$, where $\OO_K,\OO_L$ denote the ring of integers of $K,L$. We have $K \cap \OO_L = \OO_K$, and there exists $\pi \in L$ such that $\OO_L = \OO_K [\pi]$.
If $e$ denotes the ramification index of $L/K$, then $\pi^e \in K$ and it is a uniformiser of $K$.

Consider the following diagrams:
\ 

\[ \xymatrix@R=12pt{
\ulMOmega^q(\sY) \ar@{..>}[r]_-{(*)}^-{\subseteq} \ar@/^18pt/[rr] & 
\tfrac{\pi}{\pi^{e}} \cdot \ulMO_{\sY} \cdot \ulPOmega^q_{\sX} 
\ar[r]^-{\subseteq} \ar@{}[dr]|{(A)}& 
\Omega^q(\iY)
\\
\ulMOmega^q(\sX) \ar@{=}[r] \ar[u] & 
\ulMOmega^q(\sX) \ar[r] \ar[u] & 
\Omega^q(\iX) \ar[u]
} 
\]
and
\[\xymatrix{
\tfrac{\pi}{\pi^{e}} \cdot \ulMO_{\sY}^{\oplus \binom{n}{q}
} 
\ar[r] \ar@{}[dr]|{(B)}& 
L^{\oplus \binom{n}{q}} 
\\
\ulMO_{\sX}^{\oplus \binom{n}{q}} \ar[r] \ar[u] &
K^{\oplus \binom{n}{q}} \ar[u]
}\]
where all arrows except $(*)$, whose existence we prove blow, are natural inclusion maps. 
The goal is to show the outside square on the left is cartesian. The strategy is to show that the inclusion $(*)$ exists, and that the square $(A)$ is isomorphic to the square $(B)$, and that $(B)$ is cartesian.

We begin with the isomorphism $(A) \cong (B)$. By definition, Def.~\ref{defi:NC}, there exist $t_1, \dots, t_n \in \OO(\hX)$
such that the modulus is given by $t_1^{r_1} \dots t_i^{r_i}$ and 
\begin{equation} \label{equa:KOmegaBasis}
\Omega^1_{K} = K \dlog(t_1) \oplus \dots \oplus K \dlog(t_i) \oplus K dt_{i+1} \oplus \dots \oplus  K dt_{n} \cong K^n,
\end{equation}
and hence $\Omega^q \cong K^{\binom{n}{q}}$.
Similarly, 
\[
\ulPOmega^1(\sX) \cong \OO_{\hX}^{\oplus n} \subseteq K^{\oplus n},
\]
and hence $\ulPOmega^q(\sX) \cong \OO_{\hX}^{\oplus \binom{n}{q}}$ (cf. Example~\ref{exam:sncOmegaFree}).
Finally, by definition, we have 
\[
\ulMOmega^q(\sX) = \ulMO(\sX) \cdot \ulPOmega^q(\sX) \cong \ulMO(\sX) \cdot \OO_{\hX}^{\oplus \binom{n}{q}} \cong \ulMO (\sX)^{\oplus \binom{n}{q}}.
\]
Thus we obtained a natural isomorphism of the commutative diagrams $(A) \cong (B)$.

\ 

Next we prove the existence of the inclusion $(*)$. 
\begin{claim}
There are canonical inclusions 
\begin{align}
\Omega^q_{\hY} &\subseteq \pi^{1-e} \OO_{\hY} \cdot \Omega^q_{\hX} \label{eq:OOO1} \\
\ulPOmega^q_{\sY} &\subseteq \pi^{1-e} \OO_{\hY} \cdot \ulPOmega^q_{\sX} \label{eq:OOO2} \\
\ulMOmega^q_{\sY} &\subseteq \pi^{1-e} \ulMO_{\sY} \cdot \ulPOmega^q_{\sX} \label{eq:OOO3}
\end{align}
where $e$ is the ramification index, and hence $\pi^e$ is a uniformiser for $\OO_{\hX}$.
\end{claim} 

\begin{proof}
The inclusion \eqref{eq:OOO3} follows from \eqref{eq:OOO2} since $\ulMOmega^q = \ulMO \cdot \ulMOmega^q$ by definition. It remains to prove \eqref{eq:OOO1} and \eqref{eq:OOO2}. 
Notice that $\Omega^q_{\hY}$ is generated as an $\OO_{\hY}$-module by elements of the form $\omega = dy_1 \wedge \cdots \wedge dy_q$ for some $y_1,\dots,y_q \in \OO_{\hY}$. Since $\OO_{\hY} = \OO_{\hX} [\pi]$, we have $y_j = \sum_{i=0}^{e-1}a_{ij}\pi^i$ for some $a_{ij} \in \OO_{\hX}$. 
We clearly have 
$d(\sum_{i = 0}^{e-1} a_{ij} \pi^i) = \sum_{i = 0}^{e-1} a_{ij} i \pi^{i-1} d\pi + \sum_{i = 0}^{e-1} \pi^i da_{ij}$. 
Noting $d\pi \wedge d\pi = 0$, there is at most one $d\pi$ remaining in the wedge product $\omega$. Therefore, we are reduced to proving that $d\pi$ belongs to $\pi^{1-e} \OO_{\hY} \cdot \Omega^1_{\hX}$. This follows from $d\pi^e = e \pi^{e-1}d\pi$, or equivalently, $d\pi = \tfrac{1}{e} \pi^{1-e} d\pi^e$.

Finally we prove \eqref{eq:OOO2}. By Remark \ref{rema:sumPOmega}, the $\OO_{\hY}$-module $\ulPOmega^q_{\sY}$ is generated by elements of the form $\omega = \dlog y_1 \wedge \cdots \wedge \dlog y_m \wedge dy_{m+1} \wedge \cdots \wedge dy_q$ for some $y_1,\dots,y_m \in \OO_{\hY} \cap \OO_{\iY}^*$ and $y_{m+1},\dots,y_n \in \OO_{\hY}$. We want to prove $\omega \in \pi^{1-e} \OO_{\hY} \cdot \ulPOmega^q_{\sX}$. For $j=1,\dots,m$, writing $y_j = u_j \pi^{n_j}$ for some $u_j \in \OO_{\hY}^*$, we have 
\[
\dlog y_j = \dlog(u_j) + n_j \dlog(\pi) = u_j^{-1} du_j + n_j e^{-1} \dlog(\pi^e) .
\]
Since the second term belongs to $\ulPOmega^1_{\sX}$, we are reduced to the case $y_1,\dots,y_m \in \OO_{\hY}^*$ (by induction on $q$).
In this case, the element $\omega$ belongs to $\Omega^1_{\hY}$, and hence the assertion follows from \eqref{eq:OOO1}.
\end{proof}

Finally, we show that the square $(B)$ is Cartesian. 
Clearly it suffices to prove that the square 
\begin{equation} \label{sq:key1} \begin{gathered} \xymatrix{
\tfrac{\pi}{\pi^e} \cdot \ulMO_{\sY} \ar[r] & L \\
\ulMO_{\sX} \ar[r] \ar[u] & K \ar[u]
} \end{gathered}\end{equation}
is Cartesian.
The uniformiser $\pi$ induces an identification of value groups $\Gamma_{K} \subseteq \Gamma_L$ with $e\ZZ \subseteq \ZZ$. Since $\OO_K = K \cap \OO_L$ and $\OO_K^* = K \cap \OO_L^*$, and every element of $K^*$ (resp. $L^*$) can be written uniquely as $u\pi^{ei}$ (resp. $u \pi^i$) for some $i \in \ZZ$ and $u \in \OO_K^*$ (resp. $u \in \OO_L^*$), to show that $(B)$ is Cartesian it suffices to show that the induced morphism on value groups is Cartesian. Let $f$ be a local parameter for $\mX$ and write $v(f) = en$. We have 
\begin{align*}
\ulMO_\sX &= \sqrt{f \OO_K} (f \OO_K)^{-1} = \pi^e \OO_K (\pi^{en}\OO_K)^{-1} = \pi^{e - en} \OO_K, \\
\ulMO_\sY &= \sqrt{f \OO_L} (f  \OO_L)^{-1} = \pi \OO_L (\pi^{en} \OO_L)^{-1} = \pi^{1 - en} \OO_L.
\end{align*}
This shows that the image of the above square \eqref{sq:key1} with 0 removed under $v$ is 
\[ \xymatrix{
\NN {-} en {+} 1
\ar[r] & 
\ZZ \\
e\NN {-} en {+} e
 \ar[u]  \ar[r] & 
e \ZZ \ar[u] 
} \]
which is clearly Cartesian. Therefore, the original square \eqref{sq:key1} is also Cartesian.
\end{proof}

Thus the proof of Theorem \ref{thm:transfer-MOmega} is completed.

\section{Cube invariance of the modulus Hodge cohomology}\label{sec:ciOmega}

In this section, we prove that the cube invariance of the cohomology groups of the modulus Hodge sheaves.

\begin{prop}\label{prop:CI-POmega}
Let $(A,f)$ be a modulus $k$-ring with $A$ integral (and $k$ still of characteristic zero). Then the canonical sequence
\begin{equation}\label{eq:POmegaexact}
0 
\to \ulPOmega^q_{(A,f)} 
\to \ulPOmega^q_{(A[t],f)} 
\oplus \ulPOmega^q_{(A[\frac{1}{t}],f/t)} 
\to \ulPOmega^q_{(A[t, \frac{1}{t}],f)} 
\to 
0
\end{equation}
is a split short exact sequence of $A$-modules when $q=0,1$.

Suppose moreover that 
$\ulPOmega^1_{(A,f)}$ (resp. $\ulPOmega^1_{(A[t],f)}$, $\ulPOmega^1_{(A[\frac{1}{t}],f/t)}$, $\ulPOmega^1_{(A[t, \frac{1}{t}],f)}$) is flat over $A$ (resp. $A[t]$, $A[\tfrac{1}{t}]$, $A[t,\tfrac{1}{t}]$).
Then Eq.\eqref{eq:POmegaexact} is a split short exact sequence of $A$-modules for any $q \geq 0$.
\end{prop}

\begin{proof}
The case $q = 0$ is the obviously split short exact sequence
\begin{equation} \label{equa:splitAAtt}
0 \to A \to A[t] \oplus A[\tfrac{1}{t}] \to A[t, \tfrac{1}{t}] \to 0.
\end{equation}
For the case $q = 1$, we try to obtain nice decompositions of $\ulPOmega^q_{(A[t],f)} $, $\ulPOmega^q_{(A[t, \frac{1}{t}],f)}$ and $\ulPOmega^q_{(A[\frac{1}{t}],f/t)}$ respectively (see Eq.\eqref{eq:decomPOmega1}, Eq.\eqref{eq:decomPOmega2} and Eq.\eqref{eq:decomPOmega3} below).
Notice that all of these are simultaneously contained in $\Omega^1_{A[\frac{1}{f}, t, \frac{1}{t}]}$, which admits a canonical direct sum decomposition
\begin{align} \label{equa:decompAft}
\Omega^1_{A[\frac{1}{f}, t, \frac{1}{t}]} 
&\cong 
\Omega^1_{A[\frac{1}{f}]}[t, \tfrac{1}{t}]  \ \oplus\ A[\tfrac{1}{f}, t, \tfrac{1}{t}]dt
\\&\cong 
\left ( \bigoplus_{i \in \ZZ} \Omega^1_{A[\frac{1}{f}]}\  \cdot\ t^i \right ) \oplus \left ( \bigoplus_{i \in \ZZ} A[\tfrac{1}{f}]\  \cdot\ t^i dt \right ) ,
\end{align}
where we write $M[t, \tfrac{1}{t}] := M \otimes_R R[t, \tfrac{1}{t}]$, for an $R$-module $M$. To obtain the decomposition, for example, we apply \cite[04B2]{stacks-project} to the two morphisms $k \to A[\tfrac{1}{f}] \to A[\tfrac{1}{f}, t, \tfrac{1}{t}]$ together with the well-known isomorphism $\Omega_{A[\frac{1}{f}, t, \frac{1}{t}]/A[\frac{1}{f}]}^1 \cong A[\tfrac{1}{f}, t, \tfrac{1}{t}]dt$.

The left summand is included via $\Omega^1_{R} \otimes_R S \to \Omega^1_S$ and projection to the right summand is the canonical $\Omega^1_{S} \to \Omega^1_{S/R}$. In particular, the decomposition is functorial enough to induce subdecompositions (noting $d\left ( \tfrac{1}{t} \right )=-\tfrac{dt}{t^2}$)
\begin{align*}
\Omega^1_{A[t]} 
&\cong 
\Omega^1_{A}[t]  \ \oplus\ A[t]dt, \\
\Omega^1_{A[\frac{1}{t}]} 
&\cong 
\Omega^1_{A}[\tfrac{1}{t}]  \ \oplus\ A[\tfrac{1}{t}]\tfrac{dt}{t^2}, \\
\Omega^1_{A[t,\frac{1}{t}]} 
&\cong 
\Omega^1_{A}[t, \tfrac{1}{t}]  \ \oplus\ A[t, \tfrac{1}{t}]dt.
\end{align*}
By applying $\otimes_A A[\tfrac{1}{f}]$ to the above decompositions, we also obtain 
\begin{align*}
\Omega^1_{A[\frac{1}{f}][t]} 
&\cong 
\Omega^1_{A[\frac{1}{f}]}[t]  \ \oplus\ A[\tfrac{1}{f}][t]dt, \\
\Omega^1_{A[\frac{1}{f}][\frac{1}{t}]} 
&\cong 
\Omega^1_{A[\frac{1}{f}]}[\tfrac{1}{t}]  \ \oplus\ A[\tfrac{1}{f}][\tfrac{1}{t}]\tfrac{dt}{t^2}, \\
\Omega^1_{A[\frac{1}{f}][t,\frac{1}{t}]} 
&\cong 
\Omega^1_{A[\frac{1}{f}]}[t, \tfrac{1}{t}]  \ \oplus\ A[\tfrac{1}{f}][t, \tfrac{1}{t}]dt.
\end{align*}

Now we try to decompose $\ulPOmega^q_{(A[t],f)}$.
Recall that for any modulus ring $(B,b)$, by definition, $\ulPOmega^1_{(B, b)}$ is the sub-$B$-module of $\Omega^1_{B[\frac{1}{b}]}$ generated by $\Omega^1_B$ and those $\dlog b$ such that $b \in B \cap B[\tfrac{1}{b}]^*$. 
Since $A$ and therefore $A[\tfrac{1}{f}]$ is reduced, we have \[A[t] \cap A[\tfrac{1}{f}][t]^* = A \cap A[\tfrac{1}{f}]^*\]
thanks to the fact that for any commutative ring $R$, a polynomial $a_0+a_1x+\dots +a_nx^n \in R[x]$ is invertible if and only if $a_0 \in R^*$ and $a_i$ are nilpotent for all $i>0$.
Thus, one sees that the $A[t]$-module generated by $\Omega^1_{A[t]}$ and $\dlog \left ( A[t] \cap A[\tfrac{1}{f}][t]^* \right) = \dlog \left ( A \cap A[\tfrac{1}{f}]^* \right) $ is equal to $\ulPOmega^1_{(A,f)}[t]  \ \oplus\ A[t]dt$.
Therefore, we obtain a decomposition of $A[t]$-modules
\begin{equation}\label{eq:decomPOmega1}
\ulPOmega^1_{(A[t], f)} 
\cong 
\ulPOmega^1_{(A,f)}[t]  \ \oplus\ A[t]dt.
\end{equation}

Next we decompose $\ulPOmega^1_{(A[t, \frac{1}{t}],f)}$.
Since $A$ is integral by assumption, we have \[A[t, \tfrac{1}{t}] \cap A[\tfrac{1}{f}][t, \tfrac{1}{t}]^* = \bigsqcup_{i \in \ZZ} (A \cap A[\tfrac{1}{f}]^*) \cdot  t^{i}.\] 
Therefore, the image of the abelian monoid $A[t, \tfrac{1}{t}] \cap A[\tfrac{1}{f}][t, \tfrac{1}{t}]^*$ under $\dlog$ is generated by the image of $A \cap A[\tfrac{1}{f}]^*$ and the $\dlog t^i = i t^{i-1} dt$ for $i \in \ZZ$. 
Hence, we obtain a decomposition of $A[t,\tfrac{1}{t}]$-modules
\begin{equation}\label{eq:decomPOmega2}
\ulPOmega^1_{(A[t, \frac{1}{t}], f)} 
\cong 
\ulPOmega^1_{(A,f)}[t, \tfrac{1}{t}]  \ \oplus\ A[t, \tfrac{1}{t}]dt.
\end{equation}

Finally we decompose $\ulPOmega^q_{(A[\frac{1}{t}],f/t)}$.
Note that we have $A[\tfrac{t}{f}][\tfrac{1}{t}] = A[\tfrac{1}{f}][t, \tfrac{1}{t}]$ and moreover, since $A$ is integral, we have $A[\tfrac{1}{f}][t, \tfrac{1}{t}]^* = \{ a t^i : i \in \ZZ, a \in A[\tfrac{1}{f}]^*\}$.
These imply
\[
A[\tfrac{1}{t}] \cap A[\tfrac{t}{f}][\tfrac{1}{t}]^* = \bigsqcup_{i \in \Z, i \leq 0} \left ( A \cap A[\tfrac{1}{f}]^* \right ) \cdot t^i.
\]
 In particular, the image of  $A[\tfrac{1}{t}] \cap A[\tfrac{t}{f}][\tfrac{1}{t}]^*$ under $\dlog$ is generated as an abelian monoid by the image of $A \cap A[\tfrac{1}{f}]^*$ and $\dlog (\tfrac{1}{t})^i = i \dlog (\tfrac{1}{t})$ for $i \geq 0$. Hence we have a decomposition of $A[\tfrac{1}{t}]$-modules
\begin{equation}\label{eq:decomPOmega3}
\ulPOmega^1_{(A[\frac{1}{t}], \frac{f}{t})} 
\cong 
\ulPOmega^1_{(A, f)}[\tfrac{1}{t}]  \ \oplus\ A[\tfrac{1}{t}]\dlog \tfrac{1}{t}.
\end{equation}

By Eq.\eqref{eq:decomPOmega1}, Eq.\eqref{eq:decomPOmega2} and Eq.\eqref{eq:decomPOmega3}, we see that for $q = 1$, the sequence \eqref{eq:POmegaexact} in the statement is the direct sum of $\ulPOmega^1_{(A,f)} \otimes_A -$ applied to the split short exact sequence of free $A$-modules Eq.\eqref{equa:splitAAtt} and the (obviously split) short exact sequence
\begin{equation}\label{eq:obvious}
0 \to 0 \to A[t]dt \oplus A[\tfrac{1}{t}]\dlog \tfrac{1}{t} \stackrel{\sim}{\to} A[t, \tfrac{1}{t}] dt \to 0,
\end{equation}
and therefore is itself exact.

To prove the case $q \geq 1$, note that for any commutative ring $R$ and for any $R$-modules $M,N$, there exists a canonical decomposition of $R$-modules
\begin{equation} \label{equa:wedgeSum}
\bigwedge\nolimits_R^q (M \oplus N) \cong \bigoplus_{i = 0}^q \left ( \bigwedge\nolimits_R^i M \otimes_R \bigwedge\nolimits_R^{q-i} N \right ).
\end{equation}
for any $q \geq 0$. One can deduce this decomposition from $\oplus_{n \in \NN} \wedge^n$ being left adjoint to the ``degree one'' functor $A_\bullet \mapsto A_1$ from dg-$A$-algebras to $A$-modules (it commutes with sums because it's a left adjoint).
If $N$ is a free $R$-module of rank $1$ and $q \geq 1$, then Eq.\eqref{equa:wedgeSum} is simplified as
\begin{equation} \label{equa:wedgeSum-rk1}
\bigwedge\nolimits_R^q (M \oplus N) \cong  \bigwedge\nolimits_R^q M \oplus \left ( \bigwedge\nolimits_R^{q-1} M \right ) \otimes_R N.
\end{equation}
Applying Eq.\eqref{equa:wedgeSum-rk1} to the decompositions Eq.\eqref{eq:decomPOmega1}, Eq.\eqref{eq:decomPOmega2} and Eq.\eqref{eq:decomPOmega3} with $R=A[t],A[1,\tfrac{1}{t}], A[\tfrac{1}{t}]$ respectively, and by using Lem.~\ref{lem:extPOmega},
we obtain
\begin{align}
\ulPOmega^q_{(A[t], f)} 
&\cong 
\ulPOmega^q_{(A,f)}[t]  \ \oplus\ \ulPOmega^{q-1}_{(A,f)}[t]dt, \label{eq:decomPOmega4} \\
\ulPOmega^q_{(A[t, \frac{1}{t}], f)} 
&\cong 
\ulPOmega^q_{(A,f)}[t, \tfrac{1}{t}]  \ \oplus\ \ulPOmega^{q-1}_{(A,f)}[t, \tfrac{1}{t}]dt, \label{eq:decomPOmega5} \\
\ulPOmega^q_{(A[\frac{1}{t}], \frac{f}{t})} 
&\cong 
\ulPOmega^q_{(A, f)}[\tfrac{1}{t}]  \ \oplus\ \ulPOmega^{q-1}_{(A, f)}[\tfrac{1}{t}]\dlog \tfrac{1}{t}, \label{eq:decomPOmega6}
\end{align}
for any $q \geq 1$, by noting that $\bigwedge_{A[t]}^q(M \otimes_A A[t]) \cong \left (\bigwedge^q_AM\right)\otimes_A A[t]$.
Then the sequence Eq.\eqref{eq:POmegaexact} in the statement is the direct sum of $\ulPOmega^q_{(A,f)} \otimes_A -$ applied to the split short exact sequence of free $A$-modules Eq.\eqref{equa:splitAAtt} and the split short exact sequence
\begin{equation}\label{eq:decomPOmega7}
0 \to 0 \to \ulPOmega^{q-1}_{(A, f)}[t]dt \oplus \ulPOmega^{q-1}_{(A, f)}[\tfrac{1}{t}]\dlog \tfrac{1}{t} \stackrel{\sim}{\to} \ulPOmega^{q-1}_{(A, f)}[t, \tfrac{1}{t}] dt \to 0, 
\end{equation}
which is obtained by applying $\otimes_A \ulPOmega^{q-1}_{(A, f)}$ to Eq.\eqref{eq:obvious}, and therefore is itself also exact.
\end{proof}

\begin{cor}\label{cor:ci-MOmega}
Suppose that $(A,f)$ is a modulus $k$-ring with $A$ integral such that 
$\ulPOmega^1_{(A,f)}$ (resp. $\ulPOmega^1_{(A[t],f)}$, $\ulPOmega^1_{(A[\frac{1}{t}],f/t)}$, $\ulPOmega^1_{(A[t, \frac{1}{t}],f)}$) is flat over $A$ (resp. $A[t]$, $A[\tfrac{1}{t}]$, $A[t,\tfrac{1}{t}]$).
Then the canonical sequence
\begin{equation}\label{eq:MOmegaexact}
0 
\to \ulMOmega^q_{(A,f)} 
\to \ulMOmega^q_{(A[t],f)} 
\oplus \ulMOmega^q_{(A[\frac{1}{t}],f/t)} 
\to \ulMOmega^q_{(A[t, \frac{1}{t}],f)} 
\to 
0.
\end{equation}
is a split exact sequence of $A$-modules.

Consequently, for any simple normal crossing modulus pair $\sX$ (i.e., $\hX$ is regular locally Noetherian and $\mX$ is a simple normal crossing divisor), and  for $\tau \in \{\Zar,\et\}$ and $q \geq 0$, $i \in \Z$, we have
\[ H^i_{\tau}(\sX, \ulMOmega^q) \cong H^i_{\tau}(\sX {\boxtimes} \bcube, \ulMOmega^q). \]
\end{cor}

\begin{proof}
By the isomorphisms 
\[
\ulMO_{(A,f)} \otimes_A \ulPOmega^q_{(A,f)} \cong \ulMOmega^q_{(A,f)}, \quad \ulMO(A[t], f) = \ulMO(A[t], ft)
\]
of Lem.~\ref{lem:MOmega-otimes} and \cite[Lemma 5.3]{KelMi1} respectively, the split exactness of Eq.\eqref{eq:MOmegaexact} follows by applying $\ulMO_{(A,f)} \otimes_A -$ to the decompositions Eq.\eqref{eq:decomPOmega4}, Eq.\eqref{eq:decomPOmega5}, Eq.\eqref{eq:decomPOmega6} and the split exact sequence Eq.\eqref{eq:decomPOmega7}. 
The cube invariance of cohomology follows from \cite[Lemma 5.1]{KelMi1} and the quasi-coherence of $\ulMOmega^q_{\sX}$ from Cor.~\ref{cor:qcoh-MOmega}.
\end{proof}

\section{Blow-up invariance of the modulus Hodge cohomology} \label{sec:biOmega}

In this subsection, we continue to assume that the base ring $k$ is a field. 
We prove the following result.

\begin{thm}[Blow-up invariance of cohomology of $\ulMOmega$] \label{prop:bi-MOmega}
Let  $\sX \in \Sm_k$ be a modulus pair and $Z \subseteq \hX$ a closed subscheme that has normal crossings with $\mX$ in the sense of Def.\ref{defi:NC}. 
Let $f: \hY \to \hX$ be the blowup with centre $Z$, and $\mY = f^* \mX$. Then the canonical morphism in the derived category of quasi-coherent $\OO_{\hX}$-modules
\[ \ulMOmega^q_{\sX} \to Rf_*\ulMOmega^q_{\sY}\]
is an isomorphism for any $q \geq 0$.
\end{thm}

\subsection{The case of blowing-up of affine spaces}
First we prove Theorem \ref{prop:bi-MOmega} for the following special case.

\begin{prop} \label{prop:standardBLowupMOmega}
For a partition $I = \{0, \dots, n\} = \indexA \sqcup \indexB$, let
\[ \pi: \hB = Bl_{\AA^I}\AA^{\indexB} \to \AA^I = \hA \]
be the blowup of $\AA^I$ in $\AA^{\indexB} \subseteq \AA^I$ where we write $\AA^J := \Spec(k[x_j : j \in J])$ for a subset $J \subseteq I$. Given a function $r: I \to \NN_{\geq 0}$ equip $\hA$ and $\hB$ with the effective Cartier divisors 
\[ \mA = x_0^{r_0} \dots x_n^{r_n}; \qquad \qquad \mB = \pi^*\mA. \]
If the centre $\AA^{\indexB}$ of the blowup is contained in the divisor $\mA$, then the canonical comparison 
\[ \ulMOmega^q_\sA \to R\pi_*\ulMOmega^q_{\sB} \]
of complexes of quasi-coherent $\OO_{\hA}$-modules is a quasi-ismorphism, 
where $\sB = (\hB, \mB)$ and $\sA = (\hA, \mA)$.
\end{prop}

Recall the following facts from \cite[Proposition C.1, C.2]{KelMi1}.

\begin{prop}[{\cite[Prop.2.1.12]{EGAIII}, \cite[01XT]{stacks-project}}] \label{prop:projCoh}
For any ring $A$, we have
\begin{enumerate}
 \item $H^i(\PP^n_A, \OO(\ast)) = 0$ for $i \neq 0, n$.
 \item The canonical homomorphism of graded rings 
 \[ A[t_0, \dots, t_n] \stackrel{\sim}{\to} H^0(\PP^n_A, \OO(\ast)) \]
 is a bijection.
 \item 
 \[ \tfrac{1}{t_0\dots t_n}A[\tfrac{1}{t_0}, \dots, \tfrac{1}{t_n}] \stackrel{\sim}{\to} H^n(\PP^n_A, \OO(\ast)) \]
where the morphism sends an element on the left to the corresponding section of the \v{C}ech cohomology with respect to the standard covering, and the left hand side has the standard grading. In particular, the highest degree nonzero elements are $\tfrac{a}{t_0\dots t_n}$, for $a \in A \setminus \{0\}$, and these have degree $-n-1$.  
\end{enumerate}
\end{prop}

\begin{prop}[{cf.[SGA6, VII, Lem.3.5]}] \label{prop:buCohLine}
Let $k$ be a ring and write $\AA^n := \AA^n_k$ for all $n \geq 0$.
Let $f: B_{n+1} = Bl_{\AA^{n+1}}\{0\} \to \AA^{n+1}$ be the blowup of affine $(n+1)$-space at the origin, and let $\mathcal{O}(1)$ be the line bundle associated to the exceptional divisor. Set $\mathcal{O}(i) := \mathcal{O}(1)^{\otimes i}$ for all $i \in \Z$.
Then $f_*\OO(i)$ is the coherent sheaf associated to $I^i$ where $I$ is the ideal of the origin, and we set $I^i := \Gamma(\AA^{n+1}, \OO_{\AA^{n+1}})$ for $i < 0$.
Moreover, we have
\[ R^qf_*(\OO(i)) = 0 \]
for all $q > 0$ and $i > {-n-1}$.
\end{prop}

\begin{proof}[Proof of Proposition \ref{prop:standardBLowupMOmega}]
The case $q=0$ is by \cite[Proposition 4.6]{KelMi1}. We assume $q \geq 1$ in the following. 
Consider the short exact sequence of quasi-coherent $\OO_{\hB}$-modules
\[ 0 \to \pi^*\ulMOmega^q_{\sA} \to \ulMOmega^q_{\sB} \to \sC_q \to 0, \]
where the map $\pi^*\ulMOmega^q_{\sA} \to \ulMOmega^q_{\sB}$ is injective since $\ulMOmega^q_{\sA}$ is flat over $\ol{A}$ by Example~\ref{ex:locfreeomega}, and $\sC_q$ is defined to be its cokernel.
By the projection formula and Prop.~\ref{prop:buCohLine},
we have 
\[ \ulMOmega^q_{\sA} \cong \ulMOmega^q_{\sA} \otimes R\pi_*\OO_{\hB}\cong R\pi_*\pi^*\ulMOmega^q_{\sA}.    \]
Therefore, it remains to show $R\pi_*\sC_q \cong 0$.

First we consider the case $q = 1$. 
Set $M = \{ i\ |\ r_i \neq 0\}$.
Then, since the center of the blow up $\AA^T = \cap_{\nu \in N} \{x_\nu = 0\}$ is contained in $|A^\infty| = \cup_{\mu \in M} \{x_\mu = 0\}$ by assumption, we have  $M \cap N \neq \varnothing$.
We claim that the morphism 
\[ 
\pi^*\ulPOmega^1_{\sA} \to \ulPOmega^1_{\sB}
 \]
can be identified with the canonical inclusion
\newcommand{\indexC}{M}
\begin{equation} \label{equa:MOmegaAlogBInc}
\OO^{\oplus (\indexC \cup \indexB)} \oplus \OO^{\oplus (\indexC^c \cap \indexA)}
\stackrel{}{\longrightarrow}
\OO^{\oplus (\indexC \cup \indexB)} \oplus \OO(-1)^{\oplus (\indexC^c \cap \indexA)}.
\end{equation}
where $(M\cup T)^c = M^c \cap N$.
Indeed, $\hB$ admits the open covering $\{ U_j : j \in \indexA \}$ where
\[ U_j = \Spec(k[x_\tau, \tfrac{x_\nu}{x_j}, x_j \ |\  \tau \in \indexB,\ \nu \in \indexA,\ \nu \neq j ]), \]
and the restriction of $B^\infty$ to $U_j$ is the simple normal crossing divisor whose local parameter is given by 
\[
\prod_{\mu \in M} x_{\mu}
= \prod_{\tau \in M \cap T} x_{\tau} \cdot \prod_{\nu \in M \cap N} x_{\nu} 
= \left ( \prod_{\tau \in M \cap T} x_{\tau} \cdot \prod_{\nu \in M \cap N} (x_{\nu}/x_j) \right ) \cdot x_j^{|M \cap N|}.
\]
In particular, since $M \cap N \neq \varnothing$, we find that $x_j$ divides the modulus $B^\infty$, and hence $\dlog x_j$ belongs to $\ulPOmega^1_{(U_j, B^\infty \cap U_j)}$.

Thus, the sub-$\OO_{\hB}(U_j)$-module $\ulPOmega^1_{\sB} \subseteq \Omega^1_{k[x_0, \dots, x_n, x_0^{-1}, \dots, x_n^{-1}]}$ is free with $\OO_{\hB}(U_j)$-basis
\[
\left \{
\begin{array}{cc}
  \dlog x_\tau 
& \textrm{ for } \tau \in \indexC \cap \indexB\phantom{,\ \tau \neq j} \\
  \dlog \tfrac{x_\nu}{x_j} 
& \textrm{ for } \nu \in \indexC \cap \indexA,\ \nu \neq j  \\
  d x_\tau 
& \textrm{ for } \tau \in \indexC^c \cap \indexB\phantom{,\ \tau \neq j} \\ 
  d\tfrac{x_\nu}{x_j} 
& \textrm{ for } \nu \in \indexC^c \cap \indexA,\ \nu \neq j \\
 \dlog x_j
& 
\end{array}
\right.
\]
Since we have 
\[
\dlog \tfrac{x_\nu}{x_j} = \dlog x_\nu - \dlog x_j, \quad  d\tfrac{x_\nu}{x_j} = \tfrac{1}{x_j}dx_\nu - \tfrac{x_\nu}{x_j}\dlog x_j,
\] 
this is simplified as 
\begin{equation} \label{equa:MOmegaGlobalBasis}
\left \{
\begin{array}{cc}
  \dlog x_\tau 
& \textrm{ for } \tau \in \indexC \cap \indexB \\
  \dlog x_\nu 
& \textrm{ for } \nu \in \indexC \cap \indexA  \\
  d x_\tau 
& \textrm{ for } \tau \in \indexC^c \cap \indexB \\ 
  \tfrac{1}{x_j}dx_\nu 
& \textrm{ for } \nu \in \indexC^c \cap \indexA 
\end{array}
\right.
\end{equation}
and hence we have isomorphisms for all $i$,
\[
\ulPOmega^1_{(U_j,B^\infty \cap U_j)} 
\cong \bigoplus_{\mu \in M}\OO_{U_j} \dlog x_\mu \oplus 
\bigoplus_{\tau \in M^c \cap T}\OO_{U_j} dx_\tau \oplus
\bigoplus_{\nu \in M^c \cap N}\OO_{U_j} \tfrac{1}{x_j}dx_\nu ,
\]
which glue to given the desired isomorphism 
\[
\ulPOmega^1_{\sB} \cong \bigoplus_{\mu \in M}\OO_{\ol{B}} \dlog x_\mu \oplus 
\bigoplus_{\tau \in M^c \cap T}\OO_{\ol{B}} dx_\tau \oplus
\bigoplus_{\nu \in M^c \cap N}\OO_{\ol{B}} \tfrac{1}{x_j}dx_\nu.
\]
On the other hand, by Example~\ref{ex:locfreeomega}, the globally free sub-$\OO_{\hB}(U_j)$-module $\pi^*\ulPOmega^1_{\sA}$ of $\Omega^1_{\ZZ[x_0, \dots, x_n, x_0^{-1}, \dots, x_n^{-1}]}$ has basis
\begin{equation} 
\left \{
\begin{array}{cc}
  \dlog x_\mu 
& \textrm{ for } \mu \in \indexC \\
  d x_\mu 
& \textrm{ for } \mu \in \indexC^c
\end{array}
\right .
\end{equation}
Hence, the identification of 
$\pi^*\ulPOmega^1_{\sA} \to \ulPOmega^1_{\sB}$ 
with Eq.\eqref{equa:MOmegaAlogBInc}.
Thus, we have a canonical isomorphism
\[
\Coker (\pi^*\ulPOmega^1_{\sA} \to \ulPOmega^1_{\sB}) \cong \bigoplus_{\nu \in M^c \cap N} \Coker (\OO_{\ol{B}} \hookrightarrow \OO_{\ol{B}} (-1)).
\]

We move on to the analysis of $\ulMOmega^1$.
First note that the canonical comparison morphism $\pi^*\ulMO(\sA) \to \ulMO(\sB)$ can be identified with the canonical inclusion 
$\ulMO_{\sB}(i{-}1) \subseteq \ulMO_{\sB}$
where $i = |\indexC \cap \indexA| = $ the number of prime divisors of $\sA$ containing the centre $\AA^{T}$ of the blowup. 
Indeed, noting $\Bl_{\AA^I} \AA^T \cong \AA^T \times \Bl_{\AA^N} \{0\}$, we are reduced to showing the claim when $\AA^T=\{0\}$, and this case can be checked by an elementary direct computation (see the proof of \cite[Proposition 4.8]{KelMi1} for more detail).
We also have 
\[ \ulMOmega^1(\sX) = \ulPOmega^1_{\sX}   \otimes \ulMO(\sX) \]
by Lem.~\ref{lem:MOmega-otimes}.
Therefore, we can identify 
\[
\pi^*\ulMOmega^1_\sA \to \ulMOmega^1_\sB 
\]
with 
\begin{equation}\label{eq:MOmegaidem}
\ulMO_{\sB}(i-1) \otimes_{\OO_{\ol{B}}} \left ( \bigoplus_{\nu \in M^c \cap N} \OO_{\ol{B}} \right )
\to 
\ulMO_{\sB}(i-1) \otimes_{\OO_{\ol{B}}} \left ( \bigoplus_{\nu \in M^c \cap N} \OO_{\ol{B}} (-1) \right ).
\end{equation}
Therefore, we have
\[
\sC_1 \cong \ulMO_{\sB}(i-1) \otimes_{\OO_{\ol{B}}}  \left ( \bigoplus_{\nu \in M^c \cap N} \OO_{\ol{B}} (-1) /  \OO_{\ol{B}} \right ).
\]
If $M^c \cap N = \emptyset$, then we have $\sC_1 = 0$ and there is nothing to prove.
Hence we may assume $M^c \cap N \neq \emptyset$.
In particular, we have $0<i = |M \cap N| < |N|$.
Then, noting that $\sI:=\OO_{\ol{B}}(1)$ is the sheaf of ideals defining the exceptional divisor $\iota : E \subset \ol{B}$ of the blow up $\pi$, we have a canonical isomorphism $\sI^j / \sI^{j+1} \cong (\OO/\sI) (i) \cong \iota_* \iota^* \OO(i)$ for any $j \in \Z$.
Moreover, since $\ulMO_{\sA} \cong \OO_{\ol{A}}$ is a globally free line bundle, we have 
\[
\iota^* \ulMO_{\sB} = \iota^* (\ulMO_{\sA} (1-i)) \cong \OO (1-i),
\]
Thus, $\sC_1$ admits a filtration whose graded pieces are of the form
\[
\left (\sI^j \ulMO_{\sB} / \sI^{j+1} \ulMO_{\sB} \right) (-1) \cong \iota_* \iota^* \left ( \ulMO_{\sB} (j-1) \right) \cong \iota_* \OO_E (j-i)
\]
where $j=0,1,\dots,i-1$, and hence $j-i = -i,\dots,-1$.
Since $0<i<|N|$ as we observed above, we obtain 
\begin{equation}\label{equa:boundBound}
-|N| < -i < \cdots < -1 <0.
\end{equation}
Since $E \cong \PP^{|N|}$ and $\iota_*$ is exact,
the pushforward $R\pi_*(-)$ of all of these vanish by Prop.~\ref{prop:projCoh},
finishing the proof for $q=1$.

The proof for the case $q>1$ goes in the same way, with a slight modification.
Indeed, by Lem.~\ref{lem:extPOmega} and Lem.~\ref{lem:MOmega-otimes}, we have 
\[ \ulMOmega^q_\sX = \ulMO_\sX \otimes_{\OO_{\hX}} \bigwedge\nolimits^q \ulPOmega^1_{\sX}. \]
Since $\pi^*$ commutes with the exterior power, we can identify 
$\pi^*\ulPOmega^q_{\sA} \to \ulPOmega^q_{\sB}$ 
with the $q$th alternating power of the morphism Eq.\eqref{equa:MOmegaAlogBInc}. This alternating power is a sum of canonical inclusions
$ \OO \to \OO(-k)$
where now 
\begin{equation} \label{equa:MOmegaBUCohBounds}
k \in \{0, 1, \dots, |M^c \cap N|\}.
\end{equation}
Indeed, in terms of the local basis Eq.\eqref{equa:MOmegaGlobalBasis}, this $q$th exterior power has basis elements which are locally a wedge product of $q$ sections of the form $\dlog x_\mu$, $dx_\tau$, or $\tfrac{1}{x_j}dx_\nu$, and we can have at most $|M^c \cap N|$ factors of the form $\tfrac{1}{x_j}dx_\nu$. As above, tensoring with the canonical inclusion $\pi^*\ulMO_\sA \cong \ulMO_\sB(i-1) \to \ulMO_\sB$, we find that $\pi^*\ulMOmega^q_\sA \to \ulMOmega^q_\sB$ is still a direct sum of morphisms of the form 
\begin{equation}\label{eq:mor-k}
\ulMO_{\sB}(i-1) \otimes_{\OO_{\ol{B}}} \OO_{\ol{B}} 
\to 
\ulMO_{\sB}(i-1) \otimes_{\OO_{\ol{B}}}  \OO_{\ol{B}} (k),
\end{equation}
where $k \in \{0,1,\dots,|M^c \cap N|\}$.
cf. Eq.\eqref{eq:MOmegaidem}.
It suffices to show that the cokernel of Eq.\eqref{eq:mor-k} vanishes after applying $R^q\pi_*$ for $q>0$.
If $k=0$ then the assertion is trivial since the cokernel of Eq.\eqref{eq:mor-k} is $0$.
Assume $k>0$.
Then the cokernel of \label{eq:mor-k} admits a filtration whose graded pieces are of the form
\[
\left ( \sI^j \ulMO_{\sB} / \sI^{j+1} \ulMO_{\sB} \right ) (-k) \cong \iota_* \iota^* \left ( \ulMO_{\sB} (j-k) \right ) \cong \iota_* \OO_{E} (j+1-i-k),
\]
where 
\[
j+1-i-k = -(i+k-1), \dots ,-k <0.
\]
Thus, to obtain the desired vanishing, it remains to prove $i+k-1 < |N|$.
But this follows from 
\[
i+k-1 \leq |M \cap N| + |M^c \cap N| - 1= |N| - 1 < |N|.
\]
This finishes the proof.
\end{proof}

\subsection{End of proof of Theorem \ref{prop:bi-MOmega}}
Since the modulus Hodge sheaves $\ulMOmega^q$ are quasi-coherent \'etale shaves by Prop.~\ref{prop:OmegaEtaleQC}, their Zariski and \'etale cohomologies agree. 
Therefore, it suffices to show that the morphism $\ulMOmega^q_{\sX_{\et}} \to Rf_*\ulMOmega^q_{\sY_{\et}}$ is an isomorphism for any $q \geq 0$.

By the assumption that $Z$ has normal crossing with $\mX$, for any point $x \in \hX$ there exists an \'etale neighborhood $p:\hU \to \hX$ of $x$ and an \'etale morphism $q:\hU \to \AA^n = \Spec k[t_1,\dots,t_n]$ such that $p^*\mX = q^* H$ and $p^{-1}Z = q^{-1} Z_0$, where $H$ and $Z_0$ are the closed subschemes of $\AA^n$ given by
\[
H=\{\prod_{a \in A} t_a^{r_a} = 0\}, \quad Z_0 = \{ t_b =0 ; b \in B \},
\]
for some subsets $A,B \subset \{1,\dots,n\}$ and positive integers $r_a$ for $a \in A$.
Since our problem is \'etale local over $\hX$, by replacing $f$ by $f \times_{\hX} \hU$, we may and do assume from the beginning that there exists an \'etale morphism $p:\hX \to \hU$ which induces a minimal ambient morphism $f:\sX \to (\AA^n,H)$ such that $Z = f^{-1}Z_0$.
Since $\hX \to \AA^n$ is \'etale (hence flat), we obtain a cartesian diagram
\[\xymatrix{
\hY \ar[r]^{q'} \ar[d]_f \ar@{}[rd]|\square & \hY_0 \ar[d]^{f_0} \\
\hX \ar[r]_q & \AA^n
}\]
where $f_0$ is the blow-up of $\AA^n$ at $Z_0$ and $q'$ is the morphism induced by the universal property of blow-up.
By Prop.~\ref{prop:standardBLowupMOmega}, we know that the desired assertion holds for $f_0$, and hence we have 
\[
\ulMOmega^q_{(\AA^n,H)} \cong Rf_{0*} \ulMOmega^q_{(\hY_0,f_0^*H)}
\]
for all $q \geq 0$. 
By applying $q^* = Rq^*$ to the above isomorphism and by using the flat base change $q^* Rf_{0*} \cong Rf_* q^{\prime*}$ \cite[02KH]{stacks-project}, we obtain 
\[
\ulMOmega^q_{\sX}=q^* \ulMOmega^q_{\sA} \cong q^* Rf_{0*} \ulMOmega^q_{\sY_0} \cong Rf_* q^{\prime*} \ulMOmega^q_{\sY_0} = Rf_* \ulMOmega^q_{\sY},
\]
where the first and the last equalities hold by by quasi-coherence for the \'etale topology.
By using the quasi-coherence again, we deduce from the above isomorphism 
\[
\ulMOmega^q_{\sX_{\et}} \cong Rf_* \ulMOmega^q_{\sY_{\et}},
\]
finishing the proof of Theorem \ref{prop:bi-MOmega}.

\section{The modulus Hodge realisation}\label{sec:Hodge-Real}

\subsection{Representability of Hodge cohomology}

We are now ready to prove our main result Theorem \ref{thm:main-II}.

\begin{defi}
Write $\ulMSm_k^{\nc} \subseteq \ulMSm_k$ (resp. $\ulMCor_k^{\nc} \subseteq \ulMCor_k$) for the full subcategory of quasi-projective normal crossings modulus pairs. 
\end{defi}

\begin{proof}[Proof of Theorem \ref{thm:main-II}]
Fix a non-negative integer $q$.
By Thm.~\ref{thm:transfer-MOmega}, we know that $\ulMOmega^q$ defines a presheaf on $\ulMCor_k^{\nc}$. 
Since $k$ has characteristic $0$, Hironaka's resolution of singularities tells us that the inclusion functor $\ulMCor_k^{\nc} \subset \ulMCor$ is an equivalence of categories, and hence we obtain a presheaf
\[
\RMOmega^q : \ulMCor_k^{\op} \xleftarrow{\sim}  \ulMCor_k^{\nc,\op} \xrightarrow{\ulMOmega^q} \text{(abelian groups)}.
\]

We claim that $\RMOmega^q$ belongs to $\Shv_{\ulMNis} (\ulMCor_k)$. To see this, by definition, it suffices to check that for any $\sX \in \ulMCor_k$, the presheaf $\RMOmega^q_{\sX}$ on the small site $\hX_{\et}$ is a sheaf. 
For this, take a resolution of singularities $p:\hY \to \hX$ such that $\hY$ is smooth and $p^*\mX$ is normal crossing. 
Then, by definition, we have $\RMOmega^q_{\sX} = \ulMOmega^q_{(\hY,\mY)}$, and the right hand side is an \'etale sheaf on $\hY$. Noting that any \'etale covering of $\hX$ induces an \'etale covering of $\hY$, we conclude that $\RMOmega^q_{\sX}$ is an \'etale sheaf on $\hX$.

By abuse of notation, we denote the image of $\RMOmega^q$ under the natural functor $\Shv_{\ulMNis} (\ulMCor_k) \to D(\Shv_{\ulMNis} (\ulMCor_k))$ by the same symbol $\RMOmega^q$.
We will prove that $\RMOmega^q [n]$ is a $\bcube$-local object for any $n \in \Z$, that is, the functor 
\[
F(-) := \Hom_{D(\Shv_{\ulMNis}(\ulMCor_k))}(\ZZ_{\tr}(-), \RMOmega^q[n])
\]
is cube invariant. To see this, by resolution of singularities, it suffices to show that $F(\sX) \cong F(\sX \otimes \bcube)$ for any $\sX \in \ulMCor_k^{\nc}$. We fix such $\sX$ and compute:
\begin{align*}
F(\sX)
=^{1} \colim_{p:\sY \to \sX} H^n_{\Nis} (\hY, \RMOmega^q_{\sY})
=^{2} \colim_{p:\sY \to \sX; \sY \in \ulMCor^{\nc}} H^n_{\Nis} (\hY, \RMOmega^q_{\sY}),
\end{align*}
where the first colimit runs over abstract admissible blow-ups of $\sX$, and the second one runs over those such that $\sY$ is normal crossing. Also, The first equality follows from a general result from \cite[Theorem 4.6.3]{kmsy1}, and the second equality is a consequence of resolution of singularities. 
By definition of $\RMOmega^q$, we have $\RMOmega^q_{\sY} = \ulMOmega^q_{\sY}$. Moreover, by the blow-up invariance of cohomologies, Thm.~\ref{prop:bi-MOmega}, and by weak factorisation, we have  
\[
H^n_{\Nis} (\hY, \ulMOmega^q_{\sY}) \cong H^n_{\Nis} (\hX, \ulMOmega^q_{\sX})
\]
for any abstract admissible blow-ups. 
Therefore, we obtain 
\[
F(\sX) \cong H^n_{\Nis} (\hX, \ulMOmega^q_{\sX}).
\]
Since $\sX \otimes \bcube \in \ulMCor_k^{\nc}$, we also have 
\[
F(\sX \otimes \bcube) \cong H^n_{\Nis} (\hX \times \PP^1, \ulMOmega^q_{\sX \otimes \bcube}).
\]
Again by the cube invariance of cohomologies, we conclude $F(\sX) \cong F(\sX \otimes \bcube)$, which shows $\RMOmega^q$ is a $\bcube$-local object, as desired. 
Since $\ulMDM^{\eff} \subset D(\Shv_{\ulMNis} (\ulMCor_k^{\nc}))$ is identified with the full subcategory consisting of $\bcube$-local objects, we have $\RMOmega^q \in \ulMDM^{\eff}_k$.
Moreover, the above computation literally proves 
\[
\Hom_{\ulMDM^{\eff}_k} (\Z_\tr (\sX) , \RMOmega^q[n]) \cong  H^n_{\Nis} (\hX, \ulMOmega^q_{\sX}),
\]
concluding the proof of Theorem \ref{thm:main-II}.
\end{proof}

\begin{rema}\label{rem:strong-res}
In the above proof, we can replace the use of weak factorisation by a variant of Hironaka's resolution of singularities (often called strong resolution of singularities, see for example \cite[Definition 3.4 (2)]{bivariant-FV}). We thank Junnosuke Koizumi and Shuji Saito for pointing this out in a conversation with the second author. Indeed, what we have to show is that for any normal crossing $\sX$ and for any abstract admissible blow-up $p:\sY \to \sX$ with $\sY$ also normal crossing, $p$ induces an isomorphism $H_{\Nis}^n(\hX,\ulMOmega^q_{\sX}) \cong H_{\Nis}^n(\hY,\ulMOmega^q_{\sY})$. Indeed, by Hironaka's theorem, the proper birational morphism $p:\hY \to \hX$ is dominated by a string of blow-ups along normal crossing smooth centers $q:\hY_n \to \hY_{n-1} \to \cdots \to \hY_1 \to \hX$. Since the induced map $\hY_n \to \hY$ is again proper birational, we can find another string $r : \hZ_m \to \cdots \to \hZ_1 \to \hY$ factoring through $\hY_{n}$. Thus, we obtain a diagram $\hZ_m \to   \hY_n \to \hY \to \hX$, inducing a diagram of modulus pairs $\sZ_m \to  \sY_n \to \sY \to \sX$, where the modulus for the first three terms are given by pullback of $\mX$. Then, by Thm.~\ref{prop:bi-MOmega}, we know that the composites $\sZ_m \to \sY$ and $\sY_n \to \sX$ induce isomorphisms on the cohomology groups. This immediately implies that all morphisms in the above diagram, including $p$, induce isomorphisms on cohomology groups.
\end{rema}

\subsection{Modulus de Rham realisation}

Finally, we prove that the modulus Hodge realisation that we constructed \S \ref{sec:biOmega} forms a $\otimes$-triangulated functor which extends the de Rham realisation functor from \cite{LcW09}, at least on the category of geometric motives $\DM^{\eff}_{\gm}$.

Recall from \cite{kmsy3} that $\ulMDM^{\eff}_{\gm}$ is the full subcategory of $\ulMDM^{\eff}$ consisting of compact objects. 
It is proven in \cite{kmsy3} that $\ulMDM^{\eff}_{\gm}$ has a natural symmetric monoidal structure.  

\begin{thm}\label{thm:monoidal}
Suppose that $k$ has characteristic $0$, and consider the triangulated functor
\[
R_{dR} : \ulMDM^{\eff,\op}_{\gm} \to D(\mathrm{Vect}_k) ; \quad M \mapsto \Hom_{\ulMDM^{\eff}}(-,\oplus_{q \geq 0}\RMOmega^q[-q]),
\]
where $D(\mathrm{Vect}_k)$ denotes the derived category of the category of vector spaces over $k$, equipped with the monoidal structure that is induced by the usual tensor product of $k$-vector spaces. 
Then $R_{dR}$ is monoidal. 
\end{thm}

\begin{rema}\label{rem:dervec}
Note that the derived category of the category of $k$-vector spaces is naturally equivalent to the category of graded $k$-vector spaces.
Also, we use the convention that a monoidal functor $F:\sC \to \sD$ satisfies that for any objects $X,Y \in \sC$, the structure morphism $F(X) \otimes_{\sD} F(X) \to F(X\otimes_{\sC} Y)$ is an isomorphism.
\end{rema}

\begin{proof}
Let $K^b (\ulMCor)$ denote the bounded homotopy category of the additive category $\ulMCor$.
Recall from \cite{kmsy3} that the additive Yoneda embedding $\ulMCor \to \PSh (\ulMCor)$ induces a natural equivalence of categories
\[
\left [ \frac{K^b (\ulMCor)}{\langle \mathrm{CI}, \mathrm{MV} \rangle} \right ]^{\natural} \xrightarrow{\sim} \ulMDM^{\eff}_{\gm}; \quad C^\bullet \mapsto C^\bullet
\]
where the denominator on the left hand side is the thick subcategory of $K^b (\ulMCor)$ generated by cube invariance and Mayer-Vietoris triangles (see \cite[\S 3.1]{kmsy3} for more detail), and $(-)^\natural$ means taking idempotent completion. 

Let $M,N \in \ulMDM^{\eff}_{\gm}$ be any objects. 
By the universal property of $k$-modules, we obtain a natural morphism 
$\mu_{M,N}:R_{dR}(M) \otimes_k R_{dR}(N) \to R_{dR}(M \otimes N)$. Our goal is to show that $\mu_{M,N}$ is an isomorphism for any $M,N$.
Since $\mu_{M,N}$ is natural in $M,N$ and $R_{dR}$ is additive, it suffices to show that the composite 
\[
\tilde{R}_{dR}: K^b (\ulMCor)^{\op} \to \ulMDM^{\eff,\op}_{\gm} \xrightarrow{R_{dR}}  D(\mathrm{Vect}_k)
\]
is monoidal. 
Take any bounded complexes $M, N \in K^b (\ulMCor)$. 
Then, by the universal property of $k$-modules, we obtain a natural morphism 
$\mu_{M,N}:\tilde{R}_{dR}(M) \otimes_k \tilde{R}_{dR}(N) \to \tilde{R}_{dR}(M \otimes N)$. Our goal is to show that $\mu_{M,N}$ is an isomorphism for any $M,N$.
Let $m,n$ be the lengths of the bounded complexes $M,N$, respectively. We proceed by induction on $\ell=m+n$. 

When $\ell=0$ (and hence $m=n=0$), we have $M=\sX[i]$, $N=\sY[j]$ for some $\sX,\sY \in \ulMCor$ and $i,j \in \Z$. Moreover, by resolution of singularities, we may and do assume that $\sX$ and $\sY$ are normal crossing. 
Moreover, since the problem is Zariski local on $\hX$ and $\hY$, we may and do assume that $\hX = \Spec A$, $\hY = \Spec B$ are affine.
Since $\tilde{R}_{dR}$ is triangulated, we may assume $i=j=0$. Since $\tilde{R}_{dR}(\sX) =R_{dR}(\sX)= \oplus_{q \geq 0}\ulMOmega^q (\sX)[-q]$ (and similarly for $\sY$, $\sX \otimes \sY$) by Theorem \ref{thm:main-II} and by the affine-ness of $\hX,\hY$, we are reduced to showing that the natural morphism of graded $k$-modules that is induced by the universal property of the graded tensor product
\[
\mu_{\sX,\sY} : \left ( \oplus_{q \geq 0}\ulMOmega^q (\sX) \right) \otimes_k \left ( \oplus_{q \geq 0}\ulMOmega^q (\sY) \right) \to \oplus_{q \geq 0}\ulMOmega^q (\sX \otimes \sY)
\]
is an isomorphism. 

Since the problem is \'etale local on $\hX$ and $\hY$, we may write 
\[
\mX = \Spec A/(s_1^{a_1}\cdots s_m^{a_m}), \quad \mY = \Spec B/(t_1^{b_1}\cdots t_n^{b_n})
\]
for some local parameters $s_1,\dots,s_m \in A$, $t_1,\dots,t_m \in B$, and for some positive integers $a_1,\dots,a_m,b_1,\dots,b_n \in \Z_{\geq 1}$.
In this case, we have, by Example \ref{exam:sncOmegaFree}, the following identifications:
\begin{align*}
\ulMOmega^{\bullet} (\sX) &= s_1^{1-a_1} \cdots s_m^{1-a_m}\Omega^\bullet_{A/k}, \\
\ulMOmega^{\bullet} (\sY) &= t_1^{1-b_1} \cdots t_n^{1-b_n}\Omega^\bullet_{B/k}, \\
\ulMOmega^{\bullet} (\sX \otimes \sY) &= s_1^{1-a_1} \cdots s_m^{1-a_m} \cdot t_1^{1-b_1} \cdots t_n^{1-b_n} \Omega^\bullet_{A \times_k B/k}, 
\end{align*}
and hence we are reduced to proving $\Omega^\bullet_{A/k} \otimes_{k} \Omega^\bullet_{B/k} \cong \Omega^\bullet_{A\otimes_k B/k}$, where the $\otimes$ on the left hand side is the graded tensor product, but this is well-known. Indeed, for any diagram of commutative rings $A \leftarrow C \rightarrow B$, one has $\Omega^1_{A \otimes_C B/C} \cong \Omega^1_{A/C} \otimes_{C} B \oplus A \otimes_C \Omega^1_{B/C}$ by the universal property of $\Omega$. Applying $\wedge^q_{A\otimes_C B}$ to both sides of the isomorphism, one easily gets an isomorphism of graded $C$-modules $\Omega^\bullet_{A\otimes_C B} \cong \Omega^\bullet_{A/C} \otimes_C \Omega^\bullet_{B/C}$.

Next we prove the case $\ell=m+n>0$. We may and do assume $m \geq n$ without loss of generality, and hence $m >0$. Then one can find a distinguished triangle $M_0 \to M \to M_1 \to +1$ in $K^b (\ulMCor)$ such that the lengths of $M_0$ and $M_1$ are strictly smaller than $m$. Applying $\otimes N$, we obtain another distinguished triangle 
\[
M_0 \otimes N \to M \otimes N \to M_1 \otimes N \to +1.
\]
Since the maps $\mu_{M_0,N}$ and $\mu_{M_1,N}$ are isomorphisms by the induction hypothesis, the five lemma shows that $\mu_{M,N}$ is also an isomorphism, as desired.
\end{proof}

As an application of Theorem \ref{thm:monoidal}, we obtain the following.

\begin{cor} \label{cor:dualisable}
Let $T \in \ulMDM^{\eff}_{\gm}$ be an object such that $R_{dR}$ factors through $\ulMDM^{\eff}_{\gm}[T^{-1}]$.  Then $R_{dR}(T)$ is one dimensional. 
Moreover, let $M \in \ulMDM^{\eff}_{\gm}$ be an object such that $R_{dR}$ factors through $\ulMDM^{\eff}_{\gm}[T^{-1}]$ and $M$ is dualisable in $\ulMDM^{\eff}_{\gm}[T^{-1}]$. Then $R_{dR}(M)$ is finite dimensional. 
\end{cor}

\begin{proof}
By assumption, $R_{dR}(T)$ (resp. $R_{dR}(M)$) is an invertible object (resp. a dualisable object) in the category of $k$-vector spaces.
Since a $k$-vector space is invertible (resp. dualisable) precisely when it is one (resp. finite) dimensional, we are done.
\end{proof}

Finally, we check that the modulus de Rham realisation functor $R_{dR}$ is an extension of the de Rham realisation of Voevodsky's geometric mixed motives. 
Let $\mathbf{\Omega}^\bullet$ be the motivic complex constructed in \cite{LcW09}. Recall that Lecomte-Wach's de Rham realisation functor is defined by 
\[
H^q_{DR} (M) := \colim_{n} \Hom_{\DM^{\eff}} (M,\tau_{\leq n } \mathbf{\Omega}^\bullet [q]),
\]
where $\tau_{\leq n}$ denote the truncation functor. 
We regard the direct sum $H_{DR}^* := \oplus_q H_{DR}^q$ as a functor from $\DM^{\eff,\op}_{\gm}$ to the category of graded $k$-vector spaces (or the derived category of $k$-vector spaces).
Of course, for any $X \in \Sm$, this coincides with the de Rham cohomology:
\[
H^q_{DR} (M(X)) \cong H^q (X,\Omega^\bullet). 
\]
We prove that the functor $R_{dR}$ from Theorem \ref{thm:monoidal} is an extension of the de Rham realisation constructed in \cite{LcW09}, at least on the category of geometric motives. Namely:

\begin{thm} \label{thm:LWcomparison}
The composite 
\[
\DM^{\eff,\op}_{\gm} \to \ulMDM^{\eff,\op}_{\gm} \xrightarrow{R_{dR}} D(\mathrm{Vect}_k)
\]
coincides with the de Rham realisation functor $H_{DR}^*$.
\end{thm}

\begin{proof}
Let $F$ denote the composition in the assertion, and set $G:=H^\bullet_{DR}$.
Note that $F,G$ are triangulated by construction.
Let $C$ be the collection of motives of the form $M(X)$ with $X$ projective smooth.  
Recall that $\DM^{\eff}$ is compactly generated by $C$. 
In other words, any object in $\DM^{\eff}_{\gm}$ is a direct summand of an iterated extensions of direct sums of shifts of objects of $C$.
This implies that $\DM^{\eff}_{\gm}$ is equivalent to the idempotent completion of the essential image of the triangulated functor 
\[
p: K^b (\mathbf{ProjCor}) \to \DM^{\eff}_{\gm} ; \quad X \mapsto M(X),
\]
where $\mathbf{ProjCor}$ denotes the full subcategory of $\Cor$ consisting of projective smooth schemes over $k$.
Thus, we are reduced to showing that there exists an isomorphism $F\circ p \cong G\circ p$. 

Let $i_n : \mathbf{ProjCor} \to K^b (\mathbf{ProjCor})$ be the functor sending $X$ to $X[n]$ for each $n \in \Z$. 
Since any triangulated functor $H$ on $K^b (\mathbf{ProjCor})$ is determined by the compositions $H \circ i_n, n \in \Z$, it suffices to show that for any $n \in \Z$, there exists a natural isomorphism $F \circ p \circ i_n \cong G \circ p \circ i_n$. 
For any projective smooth $X \in \mathbf{ProjCor}$, and for any $n \in \Z$, we compute
\begin{align*}
F \circ p (X[-n]) &= R_{dR}(M(X[-n])) = \Hom_{\ulMDM^{\eff}} (M(X[-n]), \oplus_{q \geq 0} \RMOmega^q[-q]) \\
&\cong \Hom_{\ulMDM^{\eff}} (M(X[0]), \oplus_{q \geq 0} \RMOmega^q[n-q]) \\
&\cong \oplus_{q \geq 0} H^{n-q} (X,\Omega^{q}),
\end{align*}
and 
\[
G \circ p (X[-n]) = H^*_{DR} (M(X[-n])) \cong H^n (X, \Omega^\bullet).
\]
Since $X$ is projective smooth over a field $k$ of characteristic $0$, we have the Hodge decomposition $H^n (X, \Omega^\bullet) \cong \oplus_{q \geq 0} H^{n-q} (X,\Omega^{q})$. 

Thus, it remains to show that this decomposition forms a natural transformation on $\mathbf{ProjCor}$.
Note that the Hodge decomposition comes from the Hodge-to-de Rham spectral sequence:
\[
E_1^{i,j} = H^j (X,\Omega^i) \Longrightarrow H^{n} (X,\Omega^\bullet),
\]
which degenerates in $E_1$ terms when $X$ is projective smooth, and hence gives the Hodge decomposition. 
Therefore, it suffices to show that the above spectral sequence is functorial on $\mathbf{ProjCor}$, or even more strongly, on $\Cor$. 
To see this, take a Cartan-Eilenberg resolution $I^{*, *}$ of the de Rham complex $\Omega^\bullet$ in the category of Nisnevich sheaves with transfers $\NST$. This is possible since $\NST$ has enough injectives, and $\Omega^\bullet$ is bounded below. 
Then we obtain a spectral sequence associated to the double complex $\Gamma (X,I^{*, *})$:
\[
E_1^{i,j} = H^j \Gamma (X,I^{i,*}) \Longrightarrow H^{i+j} \mathrm{Tot} \Gamma (X,I^{*,*}) = \Hom_{D(\NST)} (X[0], \Omega^\bullet [j]), 
\]
which is functorial on $\Cor$ by construction. 
But we know that for any injective object $I \in \NST$, the restriction $I |_{\Sm}$ is flasque (and hence acyclic). Therefore, the Hodge-to-de Rham spectral sequence is induced by the double complex $I^{*,*}|_{\Sm}$, and hence it is functorial on $\Cor$, as desired. 
\end{proof}

\bibliographystyle{halpha-abbrv}
\bibliography{bib}

\printindex
\printindex[not]

%\end{landscape}
\end{document}